\documentclass[12pt]{article}
\pdfoutput=1 
\usepackage{graphicx,ae}
\usepackage{enumerate}
\usepackage{natbib}
\usepackage{url} 

\usepackage{amsfonts,amsmath,amssymb,bm}
\usepackage[top=1in, bottom=1.25in, left=1.25in, right=1.25in]{geometry}
\usepackage{amsthm}
\usepackage{subcaption}
\usepackage{wrapfig}
\usepackage{mathtools}
\usepackage{cases}
\usepackage{dsfont}
\usepackage{makecell}
\usepackage{multicol}
\usepackage{enumitem}   
\usepackage{caption}

\makeatletter
\def\maxwidth{ %
	\ifdim\Gin@nat@width>\linewidth
	\linewidth
	\else
	\Gin@nat@width
	\fi
}
\makeatother

\newcommand{\blind}{1}

\newcommand{\R}{\mathbb{R}}
\newcommand{\N}{\mathbb{N}}
\newcommand{\C}{\textbf{C}}
\newcommand{\rr}{ \|\textbf{h}\| }
\newcommand{\hh}{\textbf{h} }
\newcommand{\s}{s}
\newcommand{\aspropto}{\overset{\cdot}{\sim}}

\newcommand{\pkg}[1]{{\fontseries{b}\selectfont #1}} 

\newtheorem{theorem}{Theorem}

\newtheorem{lemma}{Lemma}
\newtheorem{remark}{Remark}

\newtheorem{corollary}{Corollary}

\allowdisplaybreaks

\begin{document}
	
\def\spacingset#1{\renewcommand{\baselinestretch}%
		{#1}\small\normalsize} \spacingset{1}

	\title{\bf Bivariate Covariance Functions of P\'olya Type }
\author{Olga Moreva$^1$\hspace{.2cm}\\
	and \\
	Martin Schlather$^2$ \\
	$^1$ Daimler AG, Werk Sindelfingen, Germany,\\ olga.moreva@daimler.com\\
	$^2$ Institute for Mathematics, University of Mannheim, Germany, \\ schlather@math.uni-mannheim.de }
\maketitle

	\begin{abstract}
		
	We provide sufficient conditions of P\'olya type which guarantee the positive definiteness of a  $2\times 2$-matrix-valued function in $\R$ and $\R^3$.
	Several bivariate covariance models have been proposed in literature, where all components of the covariance matrix are of the same parametric family, such as the bivariate Mat\'{e}rn model. Based on the  P\'olya type conditions, we introduce two novel bivariate parametric covariance models of this class, the  powered exponential (or stable) covariance model and the generalized Cauchy covariance model. Both models allow  for flexible smoothness, variance, scale, and cross-correlation parameters. The smoothness parameters are in $(0, 1]$. 	Additionally, the bivariate generalized Cauchy model  allows for distinct long range parameters.  We also show that the univariate spherical model can be generalized to the bivariate case within the above class only in a trivial way. 	In a data example on the content of copper and zinc in the top soil of Swiss Jura we compare the bivariate powered exponential model to the traditional linear model of coregionalization and the bivariate Mat\'{e}rn model.
\end{abstract}
	
	\noindent%
	{\it Keywords:}  cokriging; multivariate covariance function;  multivariate Gaussian random field; multivariate geostatistics; spatial cross-correlation.
	\vfill

{\centering		\section{Introduction}\label{sec:intro}}

	Multivariate data measured in space arise in a variety of disciplines including  soil science, ecology, mining, geology and meteorology. Air temperature and pressure in a certain geographical region or the content of two metals in a geological deposit  are examples of spatial processes with two components. Spatial dependence within and between the components is exploited in particular when the component of interest is not exhaustively sampled, whereas the measurement of other components can be easily carried out, e.g.\ in soil sciences \citep{goovaerts1999geostatistics, atkinson1992cokriging}. An appropriate multivariate spatial covariance model gives more sensible results for spatial interpolation than univariate models, see for example \cite{cressie2016multivariate}. In environmental and climate sciences it is important to model spatial meteorological data jointly in order to reflect spatial dependence within and between components adequately (see the discussions in \cite{feldmann2014}, \cite{berrocal2007}, and \cite{gel2004calibrated}); otherwise the obtained results might be unsound.

    We  focus on  a Euclidean space, $\R^n,$ $n \le 3$. Spatial data are assumed to stem from a multivariate Gaussian random field $\textbf{Z(\textbf{x}) }= (Z_1(\textbf{x}), \dots,  Z_m(\textbf{x})),$ $\textbf{x} \in \R^n,$ $m \in \N$, which is uniquely characterized by its mean and its covariance function. For simplicity, we assume in the theoretical part of the paper  that the random field has zero mean. A covariance function $\C$ of a multivariate field is a matrix-valued function, whose diagonal elements are the marginal covariance functions and the off-diagonal elements  are the cross-covariance functions. A covariance function  $\C = [C_{ij}]_{i,  j = 1}^m$ is called stationary if for any   $\textbf{x}, \textbf{h} \in \R^n$ and $i,j = 1, \dots, m$  it holds:
     \begin{align*}
      \text{cov}(Z_i(\textbf{x} + \textbf{h}), Z_j(\textbf{x}) ) = C_{ij} ( \textbf{h}).
     \end{align*}
	 $\C$ is stationary and isotropic if additionally $\C ( \textbf{h}_1) = \C( \textbf{h}_2)$ whenever $\|\textbf{h}_1 \| = \| \textbf{h}_2 \|$, i.e.\ the 
     marginal  and cross-covariance functions depend only on the distance between the locations. Hereinafter we write $\C(r)$ instead of $\C(\hh)$ with ~$r = \rr$, whenever $\C$ is stationary and isotropic.
     
     We recall that a covariance function must be positive definite, i.e.\ it guarantees that the variance of an arbitrary linear combination of observations of any involved components $Z_{i},$ $i = 1, \dots, m$, taken at arbitrary spatial locations is nonnegative. That is, for any $p \in \N,$ $\textbf{a}_1, \dots, \textbf{a}_n \in \R^m,$ and $\textbf{x}_1, \dots, \textbf{x}_p \in \R^n$ it must hold:
     \begin{equation*}
     \sum_{i=1}^{p} \textbf{a}_i^T \C(\textbf{x}_i - \textbf{x}_j ) \textbf{a}_j \ge 0.
     \end{equation*}    
     
     A comprehensive overview of covariance functions for multivariate geostatistics is found in \cite{genton2015cross} and \cite{jss}. Among these models is the linear model of coregionalization \citep{goulard1992lmc,wackernagel2003}. Although it is widely used by practitioners, it lacks flexibility; its limitations are discussed in \cite{gneiting2012matern}. 
     Models with compact support are introduced in \cite{du2013vector}, \cite{porcu2013radial} and \cite{daley2015}, see also \cite{randomFields}. \cite{kleiber2017coherence} studies the properties of multivariate random fields in the frequency domain. \cite{cressie2016multivariate} develop a conditional approach for constructing multivariate models. In this paper we restrict our attention to stationary and isotropic bivariate models, whose components stem from the same family, i.e.\ to models of the form
	
    \begin{equation}\label{eq:class}
		\C(r) =  \begin{bmatrix}
		\sigma_{1}^2 \psi_{11}(r)  & 
		\rho \sigma_{1}\sigma_{2} \psi_{12}(r ) \\
		\rho \sigma_{1}\sigma_{2} \psi_{12}(r )  &  
		\sigma_{2}^2  \psi_{22}(r )
		\end{bmatrix},
	\end{equation}	
	where  $\sigma_{i}^2> 0$ is the variance of the field $Z_i$, $ \psi_{ij}( \cdot ) =  \psi(\cdot | \bm{\theta_{ij}} , \s_{ij}) $ is a continuous univariate stationary and isotropic correlation function, which depends on a scale  parameter $\s_{ij} > 0,$ $i,j = 1,2,$ and another optional parameter $\bm{\theta_{ij}} = (\theta^{1}_{ij}, ..., \theta^{k}_{ij})$ with $k \in \N$ (e.g.\ smoothness, long range behaviour).  Necessarily, $|\rho| \le 1$. Note that isotropy implies $\psi_{12}(r)  =\psi_{21}(r)$. 
	For instance, the multivariate Mat\'{e}rn model \citep{gneiting2012matern, apanasovich2012valid} is a member of this class with 
	\begin{equation*}
		\psi\left( r | \nu, \s \right) = \frac{2^{1 - \nu}}{\Gamma(\nu)} (\s r)^{\nu} K_{\nu} (\s r), 
	\end{equation*}
	where $\s >0$ is a scale parameter, $\nu > 0$ is a smoothness parameter and $K_{\nu}$ is a modified Bessel function of the second kind. 

	The class given by \eqref{eq:class} also can  be seen as a generalization of the class of separable models introduced by \cite{mardia1993spatial}, where a multivariate covariance function factorizes into a product of a covariance matrix $\textbf{R}$ and a univariate correlation function $\psi(\cdot) $, i.e.\
	 \begin{equation*}
		 C_{ij}(r) = R_{ij} \psi(r) , \, r \ge 0, \, i, j = 1, \dots, m.
	 \end{equation*}
	That is, a separable model assumes that all components share the same spatial correlation structure  and differ only in their variances. In particular, the scale parameter is the same for both marginal and cross-covariance functions. The class \eqref{eq:class} is more flexible allowing each field to have distinct smoothness, scale, and variance parameters and admitting flexible cross-correlation between the fields. Given a univariate correlation function $\psi$, our goal is to find the parameter sets for which the function $\C$ in \eqref{eq:class} is a covariance function. Clearly, if the components are uncorrelated, i.e.\ $\rho = 0$, then $\C$ is always a bivariate covariance function. Thus, we are interested in $|\rho| > 0$. Furthermore, if $\psi_{11} = \psi_{22} = \psi_{12}$ then $|\rho| \le 1$ is also sufficient. 

	It is worth pointing out that not all univariate models can be generalized to non-trivial multivariate models in a  direct way. For example, the univariate spherical model, $\psi( r | \s ) = \left(1 - \frac{3}{2} \s  r + \frac{1}{2} (\s r)^3 \right)_+,$ $\s  > 0,$ is widely used in geostatistics, but its bivariate generalization
	\begin{equation}\label{eq:sph}
		\begin{bmatrix}
		\sigma_{1}^2  \left(1 - \frac{3}{2} \s_{11} r + \frac{1}{2} (\s_{11}r)^3 \right)_+ & 
		\rho \sigma_{1} \sigma_{2} \left(1 - \frac{3}{2} \s_{12}r + \frac{1}{2} (\s_{12} r)^3 \right)_+  \\
		\rho \sigma_{1} \sigma_{2}  \left(1 - \frac{3}{2} \s_{12} r + \frac{1}{2} (\s_{12}r)^3 \right)_+  
		& \sigma_{2}^2  \left(1 - \frac{3}{2} \s_{22} r + \frac{1}{2} (\s_{22}r)^3 \right)_+ 
		\end{bmatrix},
	\end{equation}
	with $\s_{ij} > 0,$ $|\rho| \le 1,$ $i,j = 1, 2,$ is a valid covariance model in $\R^3$ if and only if $\s_{11} = \s_{12} = \s_{22}$ or $\rho = 0$. This follows from the multivariate version of Schoenberg's theorem \citep{schoenberg1938metric, yaglom1987} and the fact that the spectral density of the spherical covariance is a pseudo periodic function with an infinite number of zeros, see Appendix \ref{app:sph} for details. 
	Of course, any convolutional approach for the cross-covariance function including both marginal covariance functions as factors is a promising candidate for a non-trivial model. Examples are given by \cite{du2013vector}, where the cross-covariance function stays constant for $r$ below a certain threshold, and the delay effect in \cite{wackernagel2003}.

	\cite{genton2015cross} pose the question, how to characterize a parameter set of the valid multivariate powered exponential (or stable) model. In Section~\ref{sec:bivariate} we give a partial answer, providing sufficient conditions for the positive definiteness of the bivariate model based on P\'olya type conditions. In a similar way we can also formulate sufficient conditions for the positive definiteness of the bivariate generalized Cauchy model. The models are flexible, intuitive and easily interpretable: in both models three parameters characterize the smoothness of the covariance functions of process components and the cross-covariance functions.  Further three parameters model the long-range behaviour in the bivariate generalized Cauchy model. The smoothness parameters of the marginal covariance functions in both models are restricted to values in $(0, 1]$, similarly to the  application of P\'olya criterion in the  corresponding univariate models.
	
In Section \ref{sec:data}, we fit a bivariate powered exponential model to the Jura dataset \citep{goovaerts1997geostatistics, pebesma2004} and compare the results with the bivariate Mat\'{e}rn model and with the linear model of coregionalization.

{\centering\section{ Flexible bivariate models of P\'olya type}\label{sec:bivariate}}

	We introduce novel bivariate covariance models of the form \eqref{eq:class} and provide sufficient conditions for their validity.  The derivation of new model classes are based on the following general result which includes a weak form of  P\'olya criterion in the univariate case as $\psi_{11}''(r) \ge 0$ implies convexity of $\psi_{11}$.

		\begin{theorem}\label{thm:scale}
		A matrix-valued function $\C$ defined by  equation \eqref{eq:class} is positive definite 
		\begin{itemize}
			\item[a)] in $\R$ if $\psi_{ij}(r),$ $i, j = 1,2,$ is continuously differentiable in $(0, \infty)$ with piecewise existing second derivative in $(0, \infty)$ and the following conditions holds
			\begin{enumerate}[label=(\roman*)]
				\item $r \psi_{ij}'(r) \to 0$ as $r \to \infty$ and $r \psi_{ij}'(r) \to 0$ as $r \to 0$,  \label{item:firstderiv}
				\item $\psi_{ij}' (r)$ is integrable in $(0, \infty)$, $i, j = 1, 2,$	 \label{item:firstderiv2}
				\item the matrix 
				\begin{equation}\label{eq:scale:matrix1}
				\begin{bmatrix}
				\psi''_{11}(r)  & 
				\rho \psi''_{12}(r ) \\
				\rho \psi''_{12}(r )  &  
				\psi''_{22}(r)
				\end{bmatrix}
				\end{equation}	
				is positive semidefinite for almost all $r \ge 0$. \label{item:pd}
			\end{enumerate}
			\item[b)] in $\R^3$ if $\psi_{ij}(r)$, $i, j = 1, 2$ is twice continuously differentiable in $(0, \infty)$ with piecewise existing third derivative in $(0, \infty)$ and the following conditions holds
			\begin{enumerate}[label=(\roman*)]
				\item  $r \psi_{ij}'(r) \to 0$, $r^2 \psi_{ij}''(r) \to 0$ as $r \to \infty$ and $r \psi_{ij}'(r) \to 0$, $r^2 \psi_{ij}''(r) \to 0$  as $r \to 0$, 	 	\label{item:firstderiv3}
				\item$\psi_{ij}' (r)$, $r\psi_{ij}''(r)$ are integrable in $(0, \infty)$, $i, j = 1, 2,$	 \label{item:firstderiv4}
				\item the matrix 
				\begin{equation}
				\begin{bmatrix}\label{eq:scale:matrix2}
				\psi''_{11}(r) - r \psi'''_{11}(r)   & 
				\rho (\psi''_{12}(r) - r \psi'''_{12}(r) ) \\
				\rho (\psi''_{12}(r) - r \psi'''_{12}(r) )  &  
				\psi''_{11}(r) - r \psi'''_{11}(r)
				\end{bmatrix}
				\end{equation}	
				is positive semidefinite for almost all $r \ge 0$. \label{item:pd2}				
			\end{enumerate}
		\end{itemize}
	\end{theorem}

	 Theorem \ref{thm:scale} as well as Theorems \ref{th:exp} and \ref{th:cauchy} below are  proven in the Appendix \ref{app:suff}.
	
	\begin{remark}
		Two following conditions are sufficient for condition \ref{item:pd} in Theorem \ref{thm:scale}, part a):
		\begin{equation}\label{eq:pospolya}
		\psi''_{ii} (r)  \ge 0,  \qquad i = 1,2, \, r \in A,
		\end{equation}
		and 
		\begin{equation}\label{eq:inf}
		\rho^2 \le \inf_{r \in A} \frac{\psi''_{11}(r) \psi''_{22}(r)}{\psi''_{12}(r) ^2},
		\end{equation}
		where $A = \{r \ge 0: \psi''_{ij} (r), i, j = 1,2 , \text{ exist}  \}$. 
		
		Two following conditions are sufficient for condition \ref{item:pd} in Theorem \ref{thm:scale}, part b):
		\begin{equation}\label{eq:pospolya2}
		\psi''_{ii} (r) - r \psi'''_{ii} (r) \ge 0,  \qquad i = 1,2,  \, r \in B,
		\end{equation}
		and 
		\begin{equation}\label{eq:inf2}
		\rho^2 \le \inf_{r \in B} \frac{(\psi''_{11}(r) - r \psi'''_{11}(r))(\psi''_{22}(r) - r \psi'''_{22}(r)) }
		{(\psi''_{12}(r) - r \psi'''_{12}(r) )^2},
		\end{equation}
		where $B = \{r \ge 0: \psi'''_{ij} (r),i, j = 1,2 , \text{ exist}  \}$. 
		
	The infimum in inequalities \eqref{eq:inf} and \eqref{eq:inf2}  is taken over all $r > 0$ with $\psi_{12}''(r) \neq 0$ and $\psi''_{12}(r) - r \psi'''_{12}(r) \neq 0$ respectively.  	Both  $\psi''_{ij}(r) \ge 0$ and $\psi_{ij}''(r) \ge  r \psi_{ij}'''(r)$ hold true for  completely monotone 	$\psi_{ij}(r),$  $i, j = 1, 2.$
	\end{remark}
	
	\subsection{Bivariate powered exponential model}\label{sec:bistable}
	The univariate powered exponential correlation function
	\begin{equation*}
	\psi( r|  \alpha, s ) = \exp(-(s r)^{\alpha}),
	\end{equation*}
	$s >0,$ $\alpha \in (0, 2]$,  contains the exponential model ($\alpha=1$)  and the Gaussian model $(\alpha = 2)$. It permits the full range of allowable values for the fractal dimension \citep{gneiting2002compactly}. Unlike the Mat\'{e}rn model, the univariate powered exponential correlation function does not allow for a smooth parametrization of the  differentiability of the field paths. Indeed, the paths are continuous and non-differentiable for $\alpha <2$ and  infinitely often differentiable for $\alpha = 2$.  Nevertheless, the powered exponential covariance 
	may be a good alternative for non-differentiable fields due to its simplicity. The univariate powered exponential covariance is used in 
	\cite{guillot2009computer}, \cite{henderson2002modeling}, and \cite{kent1997estimating}, for example.
	
	According to \eqref{eq:class}, the marginal covariance functions of the bivariate powered exponential model,
	\begin{equation}\label{eq:bistable1}
		\begin{aligned}
		C_{11}(r) = \sigma_1^2 \exp(-(\s_{11} r)^{\alpha_{11}}), \\
		C_{22}(r) = \sigma_2^2 \exp(-(\s_{22} r)^{\alpha_{22}}),
		\end{aligned}
	\end{equation}
    are	of powered exponential type with variance parameter $\sigma_{i}$, smoothness parameter $\alpha_{ii} \in (0, 2]$ and scale parameter $\s_{ii} > 0$, $i = 1, 2;$ the cross-covariance functions,
	\begin{align}\label{eq:bistable2}
	C_{12}(r) = C_{21}(r)= \rho  \sigma_1 \sigma_2 \exp(-(\s_{12} r)^{\alpha_{12}}), 
	\end{align}
	are also a powered exponential function with colocated correlation $\rho$, $|\rho | \le 1,$ smoothness parameter $\alpha_{12} \in (0, 2]$ and scale parameter $\s_{12} > 0$. 
	
	Whilst Theorem \ref{th:exp} below will give a sufficient condition for the positive definiteness of  the powered exponential model with $\alpha_{ii} \in (0, 1],$ $i = 1, 2,$ the following two corollaries of Schoenberg's theorem provide a necessary and sufficient condition for the special cases $\alpha_{ij} = 1$ and $\alpha_{ij} =2$, respectively for $i, j = 1, 2.$ 
	The bivariate exponential model is a special case of the bivariate  Mat\'{e}rn model, the calculations for $\rho^2$ boundaries follow directly from Theorem 3 in \cite{gneiting2012matern}. 
	\begin{corollary}\label{cor:biexp}
	The  bivariate exponential model defined by \eqref{eq:bistable1} and \eqref{eq:bistable2} with $\alpha_{11} = \alpha_{12} = \alpha_{22} = 1$ is a covariance function in $\R^n,$ $n \in \N,$ if and only if
	\begin{equation}\label{eq:bivexp}
		\rho^2 \le \frac{s_{11} s_{22}}{s_{12}^2} \inf_{r>0} \frac{(s_{12}^2 + r^2 )^{1+n} }
		{ (s_{11}^2 + r^2)^{1/2 + n/2}(s_{22}^2 + r^2)^{1/2 + n/2}}.
	\end{equation}
	In particular, this can be written as one of the following cases:
	\begin{enumerate}
		\item if $s_{12} \le \min \{s_{11}, s_{22} \} $  the bivariate exponential model is valid if and only if 
		\begin{equation*}
			\rho^2 \le  \left( \frac{s_{12}^2}{s_{11} s_{22}} \right)^n 
		\end{equation*}
		\item if $\min \{s_{11}, s_{22} \} \le s_{12} \le \max \{s_{11}, s_{22} \} $  the infimum in \eqref{eq:bivexp}
		is attained either if $r = 0$, or in the limit as $r \to \infty,$ or if 
		\begin{equation*}
			r^2 = \frac{s_{11}^2 s_{12}^2 + s_{12}^2 s_{22}^2 - 2s_{11}^2s_{22}^2 }{ s_{11}^2 + s_{22}^2 - 2 s_{12}^2 }.
		\end{equation*}
		\item if $s_{12} \ge \max \{s_{11}, s_{22} \} $  the bivariate exponential model is valid if and only if 
		\begin{equation*}
			\rho^2 \le \frac{s_{11} s_{22}}{s_{12}^2}.
		\end{equation*}
	\end{enumerate}
\end{corollary}

\begin{corollary}\label{cor:bigauss}
	The bivariate Gaussian model defined by \eqref{eq:bistable1} and \eqref{eq:bistable2} with $\alpha_{11} = \alpha_{12} = \alpha_{22} = 2$ is a covariance function in $\R^n$ if and only if one of the following conditions holds
	\begin{enumerate}[label=(\roman*)]
		\item $\s_{12}^2 \le 2 \s_{11}^2 \s_{22}^2/(s_{11}^2 + s_{22}^2) $ and $\rho^2  \le  (\s_{12}^{2}/ (\s_{11}\s_{22}))^n$
		\item  $\rho = 0$. 
	\end{enumerate}
\end{corollary}	

Now we consider equations \eqref{eq:bistable1} and \eqref{eq:bistable2}  with  $\alpha_{ii} \in (0, 1],$ $i = 1,2,$  and we define auxiliary  functions $q_{\alpha, \s}^{(n)}(r)$, $n \in \{1, 3\},$ by
	 \begin{align*}
		 q_{\alpha, \s}^{(1)}(r) & =   \alpha (\s r)^{\alpha}  - \alpha + 1, \\
		 q_{\alpha, \s}^{(3)}(r) & =  \alpha^2 (\s r)^{2\alpha} + \alpha (4 - 3 \alpha)  (\s r)^{\alpha}  + \alpha^2 - 4 \alpha + 3 .
	 \end{align*}

	\begin{theorem}\label{th:exp}
		A matrix-valued function $\C$ given by equations~\eqref{eq:bistable1} and \eqref{eq:bistable2} with $\alpha_{ii} \in (0, 1],$ $i = 1, 2,$ and $\alpha_{12} \in (0, 2]$ is a covariance model in $\R^n$, $n \in \{1, 3 \}$, if 
	\begin{equation}\label{eq:exprho}
	\rho^2 \le  \frac{\alpha_{11} \alpha_{22} \s_{11}^{\alpha_{11}} \s_{22}^{\alpha_{22}}}{\alpha_{12}^2 \s_{12}^{2\alpha_{12}} }
	\inf_{r > 0}
	r^{ \alpha_{11} + \alpha_{22} - 2\alpha_{12}} \exp\left(2 (\s_{12}r)^{\alpha_{12}} -(\s_{11}r)^{\alpha_{11}} - (\s_{22}r)^{\alpha_{22}} \right)  \frac{q_{\alpha_{11}, \s_{11}}^{(n)}(r)  q_{\alpha_{22}, \s_{22}}^{(n)}(r)  }{ (q_{\alpha_{12}, \s_{12}}^{(n)}(r) )^2}, 
	\end{equation}
	In particular,  the infimum  in \eqref{eq:exprho} is positive if and only if one of the following conditions is satisfied
	\begin{enumerate}[label=(\roman*)]
	\item  $\alpha_{12} = \alpha_{11} = \alpha_{22}$ and  $\s_{12}^{\alpha_{11}} \ge (\s_{11}^{\alpha_{11}} + \s_{22}^{\alpha_{11}})/2$,
	\item  $\alpha_{12} = \alpha_{11} > \alpha_{22}$ and $\s_{12} > 2^{-1/\alpha_{11}} \s_{11}$,
	\item  $\alpha_{12} = \alpha_{22} > \alpha_{11}$ and   $\s_{12} > 2^{-1/\alpha_{22}} \s_{22}$,
	\item $\alpha_{12} > \max\{\alpha_{11}, \alpha_{22}\}$.
	\end{enumerate}
	Moreover,
	if  $\alpha_{12} <  (\alpha_{11} + \alpha_{22})/2$  the model is valid only for $\rho = 0$.
	\end{theorem}
	
	As inequality \eqref{eq:exprho}  provides only a sufficient but not a necessary condition for positive definiteness, zero infimum in inequality \eqref{eq:exprho} does not imply that the model defined by \eqref{eq:bistable1} and \eqref{eq:bistable2} is not a valid covariance model. 

	\if1\blind {The model will be implemented  in R package RandomFields \citep{randomFields}.} \fi
	Figure \ref{fig:expRho} provides an example of the maximum attainable $|\rho|$ in inequality~\eqref{eq:exprho} that has been found numerically. 

	\begin{figure}
	\begin{minipage}[b]{0.45\textwidth}
		\includegraphics[width=\textwidth]{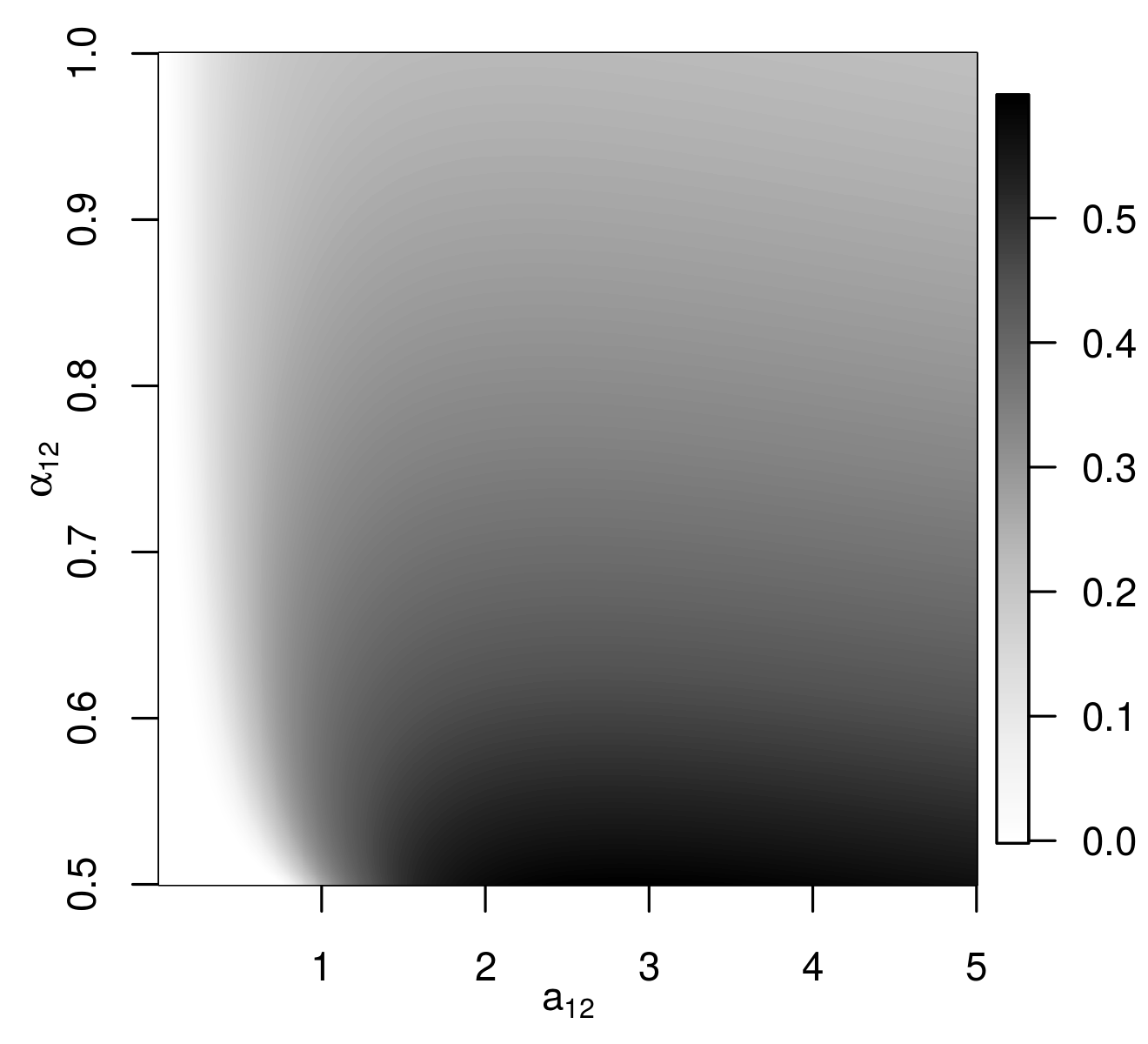}	
		\caption{The maximum attainable $|\rho|$  in inequality \eqref{eq:exprho} for the bivariate powered exponential covariance model in $\R$. The parameters are $\sigma_{1} = \sigma_{2} = 1$, $\alpha_{11} = 0.2,$ $\alpha_{22} = 0.5$, $\s_{11} = 2,$ $\s_{22} = 3$.  }
		\label{fig:expRho}
	\end{minipage}
	\hfill
	\begin{minipage}[b]{0.45\textwidth}
		\includegraphics[width=\textwidth]{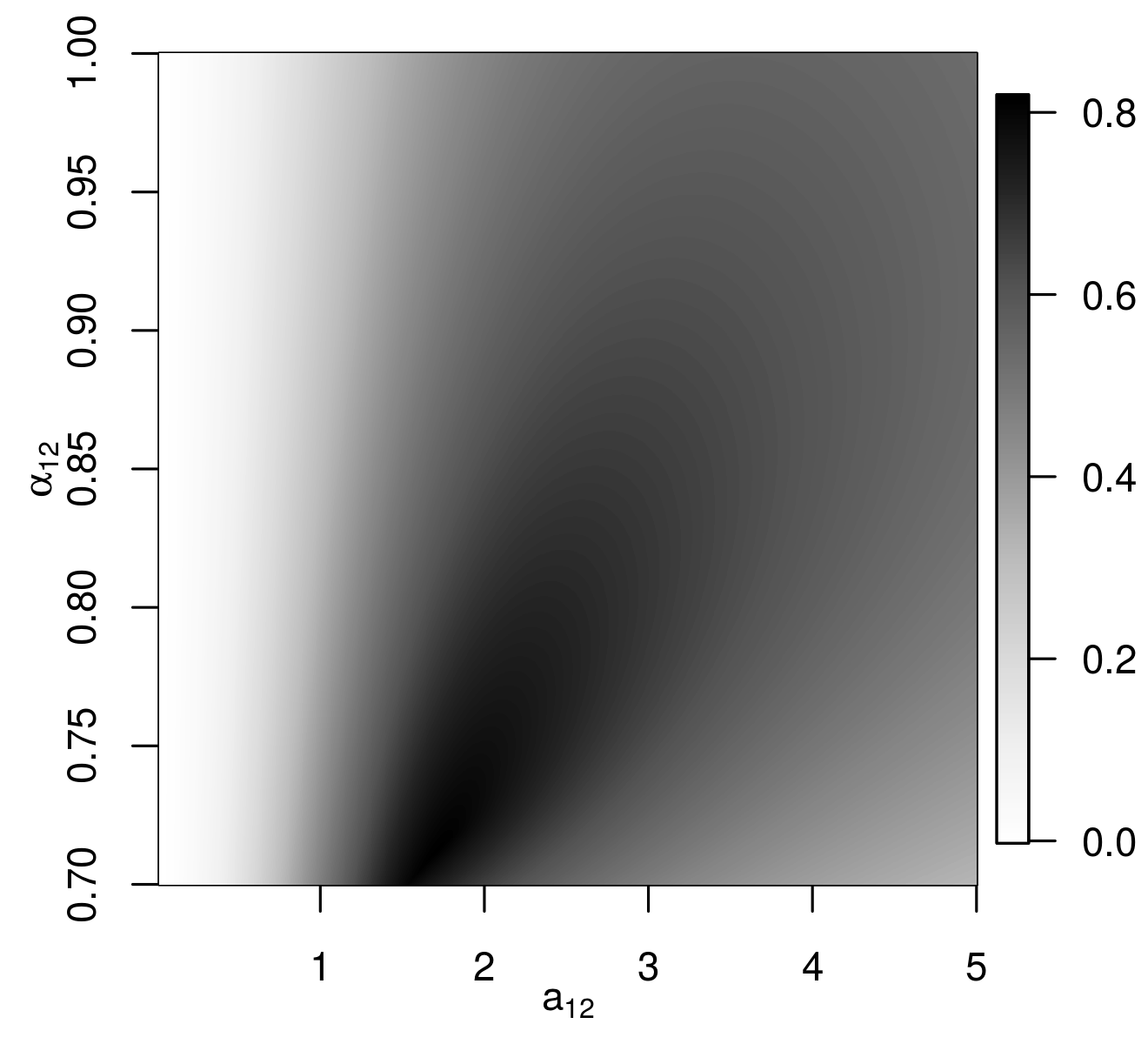}	
		\caption{The maximum attainable $|\rho|$ in inequality \eqref{eq:cauchyrho} for the bivariate Cauchy covariance model in $\R$. The parameters are  $\sigma_{1} = \sigma_{2} = 1$, $\alpha_{11} = 0.5,$ $\alpha_{22} = 0.9$, $\beta_{11} =2,$ $\beta_{12}  = 2.5,$ $\beta_{22}  = 2.1$,  $\s_{11} = 2,$ $\s_{22} = 2.5$. }
		\label{fig:CauchyRho}
	\end{minipage}
	\end{figure}

	\subsection{Bivariate generalized Cauchy model}\label{sec:bicauchy}
	The univariate generalized Cauchy model, $$\psi( r | \alpha, \beta, \s) = (1 + (\s r)^{\alpha} )^{-\beta/\alpha}, $$ has been introduced  in \cite{gneiting.powerlaw} and \cite{cauchy}.  Here $\s >0$ is a scale parameter, $\alpha \in (0, 2]$ is a smoothness parameter and $\beta > 0$ controls the long range behaviour of the field.

	Marginal covariance functions of the bivariate generalized Cauchy model,
	\begin{equation}\label{eq:biacauchy1}
	\begin{aligned}
	C_{11}(r) = \sigma_1^2 (1 + (\s_{11} r)^{\alpha_{11}} )^{-\beta_{11}/\alpha_{11}}, \\
	C_{22}(r) = \sigma_2^2 (1 + (\s_{22} r)^{\alpha_{22}} )^{-\beta_{22}/\alpha_{22}},
	\end{aligned}
	\end{equation}
	 are of generalized Cauchy type with variance parameter $\sigma_{i} >0$,  smoothness parameter $\alpha_{ii} \in (0, 2]$, long range parameter $\beta_{ii} > 0$ and scale parameter $\s_{ii} > 0$, $i = 1, 2.$ Each cross-covariance,
	\begin{align}\label{eq:biacauchy2}
	C_{12}(r) = C_{21}(r) = \rho \sigma_1 \sigma_2 (1 + (\s_{12} r)^{\alpha_{12}} )^{-\beta_{12}/\alpha_{12}},
	\end{align}
	is also  of generalized Cauchy type  with colocated correlation $\rho$, $|\rho| \le 1,$ smoothness parameter $\alpha_{12} \in (0, 2]$, long range parameter $\beta_{12} > 0$ and scale parameter $\s_{12} > 0$.
	
	We define the auxiliary functions $p_{\alpha, \beta, \s}^{(n)}(r)$, $n \in \{1, 3\}$,
	\begin{align*}
	p_{\alpha, \beta, \s}^{(1)}(r) & = \frac{(\beta + 1) (\s r)^{\alpha}  -\alpha + 1  }{(1 + (\s r)^{\alpha} )^{\beta/\alpha + 2}}, \\
	p_{\alpha, \beta, \s}^{(3)}(r) & = \frac{ (\beta + 1)(\beta + 3)(\s r)^{2\alpha} + (4\beta + 6 - 4 \alpha - 3\beta\alpha - \alpha^2) (\s r)^{\alpha} + (\alpha -1 )(\alpha -3) }{(1 + (\s r)^{\alpha} )^{\beta/\alpha + 3}}.
	\end{align*}
	
	\begin{theorem}\label{th:cauchy}
		A matrix-valued function $\C$ given by equations \eqref{eq:biacauchy1} and \eqref{eq:biacauchy2} with $\alpha_{ii} \in (0, 1],$  $\alpha_{12} \in (0, 2]$ and $\beta_{ij} > 0,$ $i, j = 1, 2,$  is a covariance function in $\R^n,$ $n \in \{1, 3\},$ if
		
		\begin{equation}\label{eq:cauchyrho}
		\rho^2 \le
		\frac{\beta_{11} \beta_{22}}{\beta_{12}^2} \frac{\s_{11}^{\alpha_{11}} \s_{22}^{\alpha_{22}}}{ \s_{12}^{2\alpha_{12}}}  \inf_{r > 0} r^{\alpha_{11} + \alpha_{22} - 2\alpha_{12} } \frac{p_{\alpha_{11}, \beta_{11} , \s_{11} }^{(n)}(r) p_{\alpha_{22}, \beta_{22} , \s_{22}}^{(n)}(r)}{(p_{\alpha_{12}, \beta_{12} , \s_{12}}^{(n)}(r))^2}
		\end{equation}
		
			In particular, 
	\begin{enumerate}[label=(\roman*)]
			\item if $\alpha_{12} \ge  (\alpha_{11} + \alpha_{22})/2$ and $\beta_{12} \ge (\beta_{11} + \beta_{22})/2$  the infimum in inequality \eqref{eq:cauchyrho} is positive; \label{item:cauchy6}
			\item if $\alpha_{12} < (\alpha_{11} + \alpha_{22})/2$  the model is valid if and only if $\rho = 0$; \label{item:cauchy1}
			\item if $\beta_{12} <  ( \min \{ \beta_{11}, n  \}  +  \min \{ \beta_{22}, n  \})/2$, the model is valid if and only if $\rho = 0$;  \label{item:cauchy2}
			 \item if $\beta_{12} < (\beta_{11} + \beta_{22})/2$, the infimum in inequality \eqref{eq:cauchyrho}  is zero.  \label{item:cauchy5}
	\end{enumerate} 
	\end{theorem}
	Analogously to the powered exponential model, inequality \eqref{eq:cauchyrho}  is  only a sufficient but not a  necessary condition for positive definiteness. Figure \ref{fig:CauchyRho} provides an example of  the maximum attainable $|\rho|$  in inequality \eqref{eq:cauchyrho} that has been found numerically.

\begin{remark}
	 Note that the inequalities \eqref{eq:pospolya} and \eqref{eq:pospolya2}  must hold only for diagonal covariance functions, but not for the cross-covariance function. 
	This allows $\alpha_{12}$  to take values in $(0, 2]$ in the bivariate powered exponential model and the bivariate generalized Cauchy model. 
\end{remark}

{\centering\section{Data example:  content of copper and zinc in Swiss Jura}\label{sec:data}}
	
	\begin{figure}
		\includegraphics[width=\maxwidth]{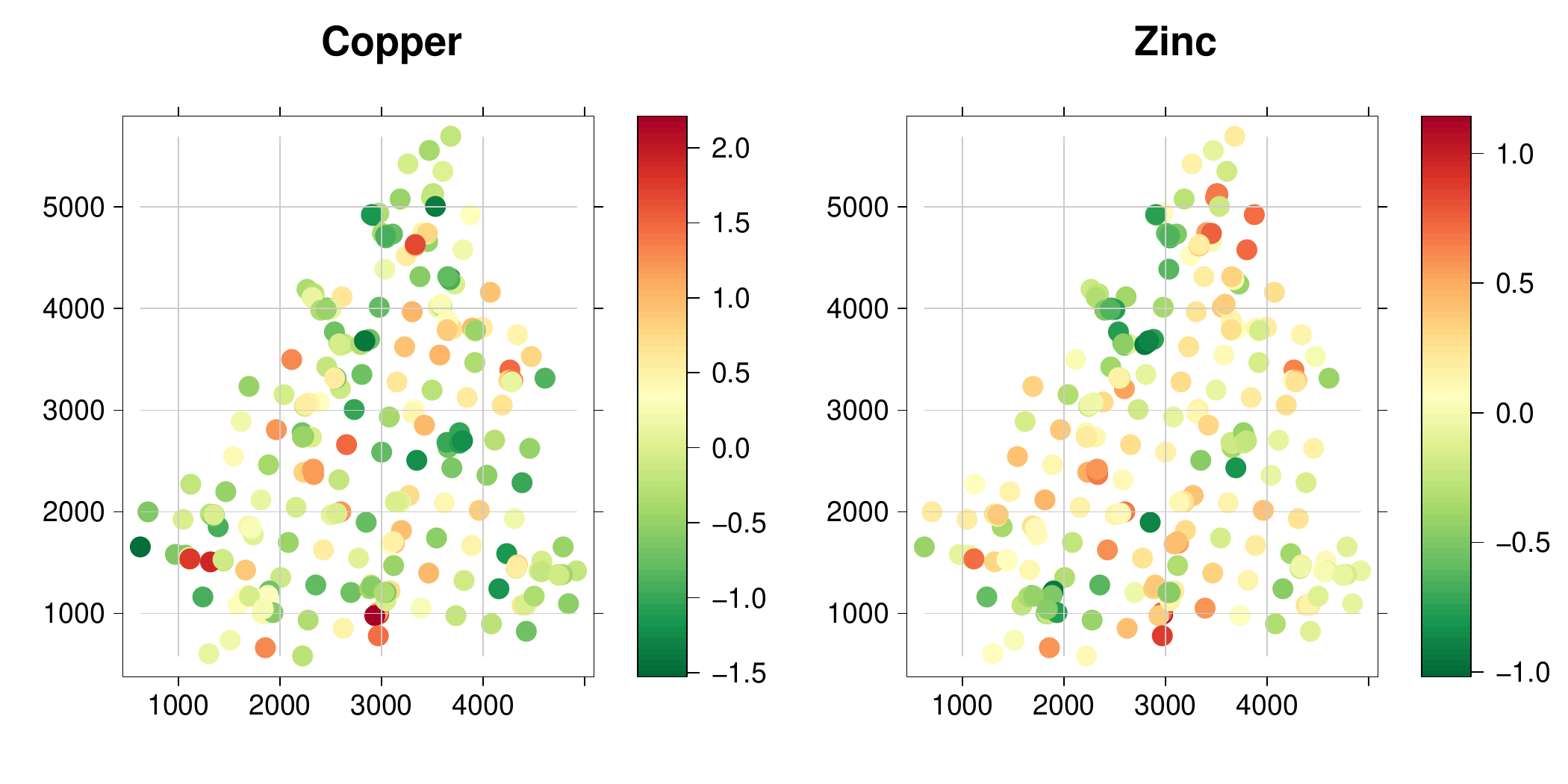} 	
		\caption{Concentration of copper and zinc in the topsoil.}\label{fig:jura}
	\end{figure}
	
	The classical geostatistical  dataset Jura from Pierre Goovaerts' book \citep{goovaerts1997geostatistics}  is provided by the package \pkg{gstat} (\cite{pebesma2004}, \cite{graler2016}). It contains concentrations of seven heavy metals (cadmium, cobalt, chromium, copper, nickel, lead and zinc)  in the topsoil  of the 14.5 km$^2$ region in Swiss Jura.  In this section we analyze the measurements of copper and zinc.  The measurements were sampled on a square grid at 250 m intervals with additional nesting with distances of 100 m, 40 m, 16 m and 6 m \citep{webster1994coregionalization}. The basic grid consists of 207 nodes, out of which 38 nodes were selected for nesting. Starting from each of these 38 nodes, the first location was chosen 100 m away in a random direction. The second location was chosen 40 m away from the first one again in a random direction. In a similar way the third and the forth locations were picked out, see Figure~\ref{fig:jura} for the arrangement of the locations. For more details on the sampling scheme and its statistical impact see \cite{atteia1994geostatistical}, \cite{webster1994coregionalization} and Chapters 2.3.1 and 4.1.1 in \cite{goovaerts1997geostatistics}. The content of zinc and copper is measured in parts per million (ppm), which means that the data are compositional and range from 0 to $10^6$. However, since the concentrations of copper and zinc are low (maximum 166.4 ppm for copper and 259.8 ppm for zinc), we analyze the dataset in a non-compositional way, following \cite{pebesma2017meuse} and \cite{goovaerts1997geostatistics}, rather then employ a compositional approach (\cite{aitchison1982statistical}, \cite{pawlowsky2011compositional}, \cite{pawlowsky2006compositional}). 
	
	The measurements at 359 locations are divided into a training set (259 locations) and a validation set (100 locations). The training set consists of grid points and the nested points, while the validation set contains only grid points. Exemplarily we fit the bivariate powered exponential model, the bivariate Mat\'{e}rn model and the linear model of coregialization (LMC)  to the training set and compare the models performance on the validation set.
	
	Following \cite{webster1994coregionalization} we first take the log-transform of the metals concentration and then subtract the mean values of the logarithms. Figure \ref{fig:jura} shows the transformed concentrations of copper and zinc.  To asses the normality of the data, we examine one and two dimensional distributions.  Shapiro-Wilk test does not reject the hypothesis that  marginal distributions of zinc and copper are univariate normal at significance level 0.05. QQ-plots in Figures \ref{fig:qqplotcopper} and \ref{fig:qqplotzinc} for marginal distributions of copper and zinc also suggest that they are close to normal. The chi-squared QQ-plot in Figure \ref{fig:qqplotQQchisq} does not go against the bivariate normal distribution of the colocated data, neither rejects the Royston's test the bivariate normality at significance level 0.05. Henceforth we assume that the data stem from a bivariate Gaussian process with zero mean.
	\begin{figure}
		\begin{minipage}[b]{0.32\textwidth}
			\includegraphics[width=1\maxwidth]{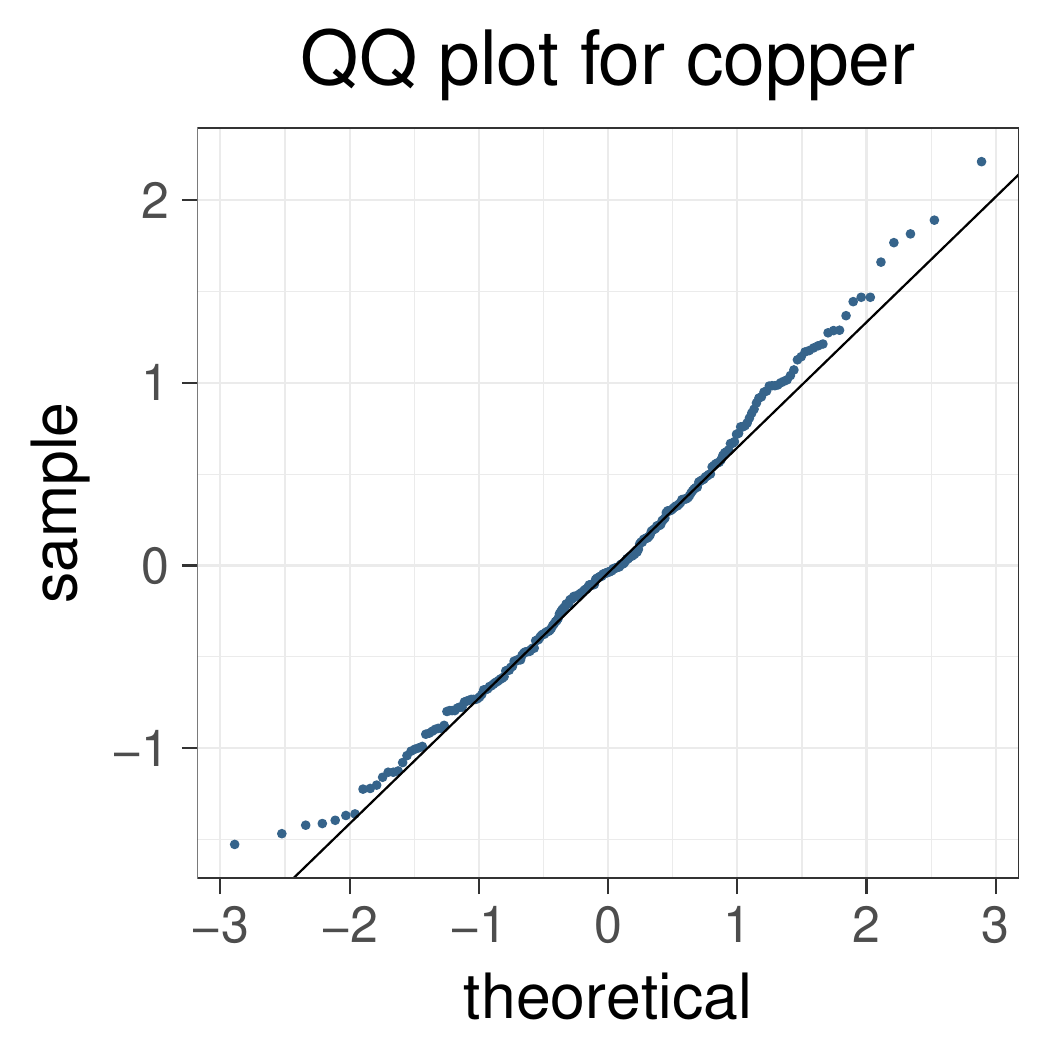} 	
			\caption{QQ plot for  copper concentrations \newline}\label{fig:qqplotcopper}
		\end{minipage}
		\hfill
		\begin{minipage}[b]{0.32\textwidth}
			\includegraphics[width=1\maxwidth]{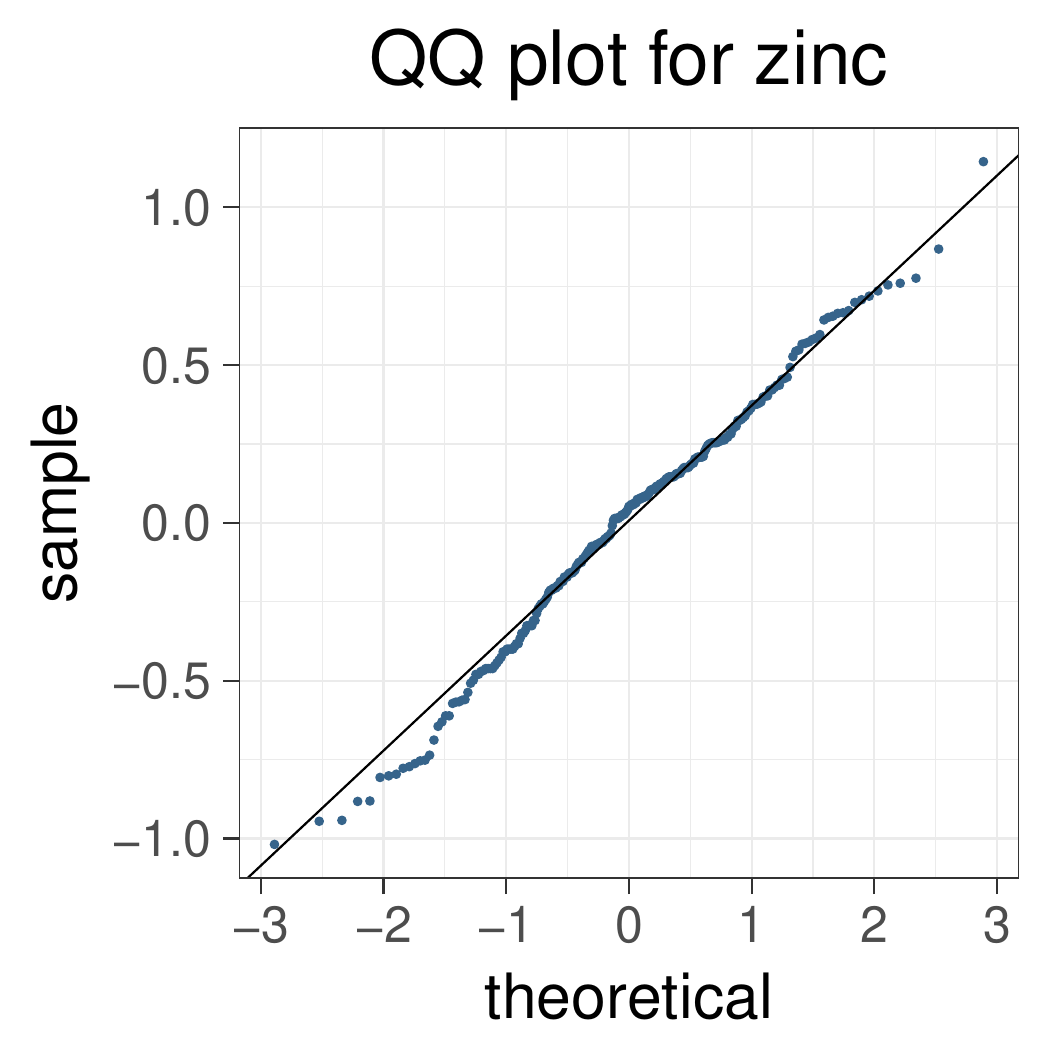} 	
			\caption{QQ plot for  zinc concentrations \newline}\label{fig:qqplotzinc}
		\end{minipage}	
		\hfill	
		\begin{minipage}[b]{0.32\textwidth}
			\includegraphics[width=1\maxwidth]{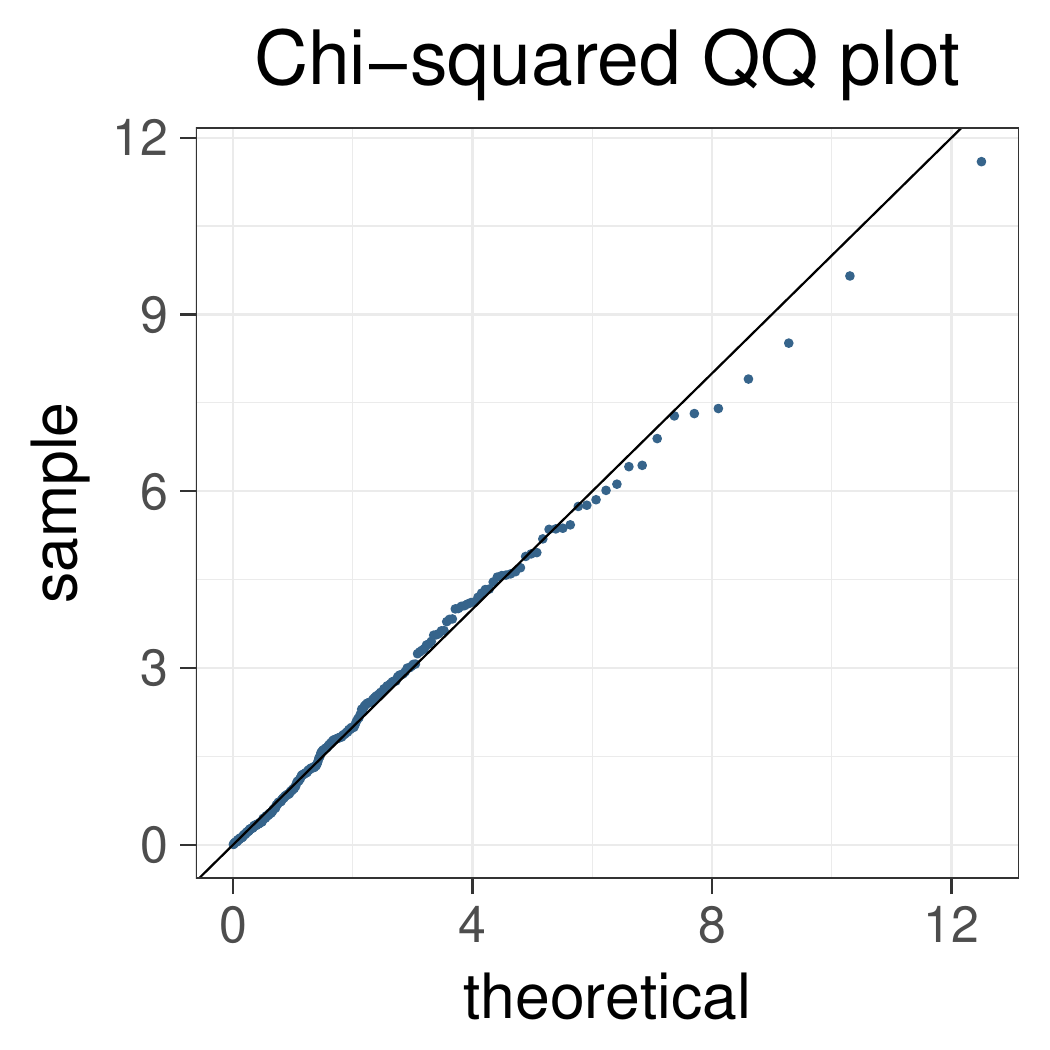} 	
			\caption{Chi-squared QQ-plot  for copper and zinc concentrations}\label{fig:qqplotQQchisq}
		\end{minipage}		
	\end{figure}
	
	The colocated empirical correlation of the data is 0.62, therefore it is reasonable to fit a bivariate covariance model. Covariance functions, which are not differentiable at the origin, are often used in geostatistics, see for example \cite{goovaerts1999geostatistics}, \cite{journel1974geostatistics}, \cite{lark2006spatial}, \cite{oliver2014tutorial}. Before fitting bivariate covariance models to the data, we fit a univariate powered exponential model to copper  and zinc observations separately in order to see if the condition $\alpha_{ii} \in (0, 1],$ $i = 1,2,$  in the bivariate powered exponential covariance model  is restrictive for this dataset. To account for measurement error we add the nugget effects to the univariate powered exponential models 
	\begin{align*}
	C_{C}(r) & = \sigma_{C}^2 \exp\left(-(\s_{C}r)^{\alpha_{C}}\right) + \tau_{C}^2  \mathds{1}(r =0),\\
	C_{Z}(r) & = \sigma_{Z}^2 \exp\left(-(\s_{Z}r)^{\alpha_{Z}}\right) + \tau_{Z}^2  \mathds{1}(r =0),
	\end{align*}
	where $r > 0$, $\alpha_{C}, \alpha_{Z} \in (0, 2]$,  and $\sigma_{C}, \sigma_{Z},\tau_{C}, \tau_{Z},\s_{C} ,\s_{Z}> 0$. Subscripts $C$ and $Z$ refer for copper and zinc, respectively.
	The maximum likelihood estimates of parameters for the univariate powered exponential model applied to the copper and zinc data are shown in the first line of Table \ref{tb:MLEbistable}. The fit suggests that the smoothness parameters $\alpha_C$  and $\alpha_Z$ for copper  and zinc,  respectively, are less than one. Copper and zinc have different scale parameters, $1/s_{C} = 94.8$ and $1/s_{s_Z} = 188.6$, therefore a flexible bivariate model is needed. In  our full bivariate powered exponential covariance model the diagonal elements are
	\begin{align*}\label{eq:copperzincmarg}
	C_{C}(r) & = \sigma_C^2 \exp\left(-(s_C r)^{\alpha_C}\right) + \tau_C^2  \mathds{1}(r =0) , \\
	C_{Z}(r) & = \sigma_Z^2 \exp\left(-(s_Z r)^{\alpha_Z}\right) + \tau_Z^2  \mathds{1}(r =0).
	\end{align*}
	On the off-diagonal we have
	\begin{equation*}\label{eq:copperzinccross}
	C_{CZ}(r)= C_{ZC}(r) = \rho \sigma_C \sigma_Z \exp\left(-(s_{CZ}r)^{\alpha_{CZ}}\right),
	\end{equation*}
	where $\alpha_C, \alpha_Z \in (0, 1]$, $\alpha_{CZ} \in (0, 2]$, $s_C, s_Z, s_{CZ} >0$ and $|\rho| \le 1$ satisfy the conditions of Theorem \ref{th:exp} and $\sigma_C, \sigma_Z, \tau_C, \tau_Z >0$.
	\begin{table}
		\centering
		\caption{Maximum likelihood estimates of parameters for the bivariate powered exponential model applied to the  copper and zinc data.}
		\label{tb:MLEbistable}
		\begin{tabular}{lccccccccccc}
			\hline
			\hline
			Model & $\sigma_C$ & $\sigma_Z$ &  
			$\alpha_C$ & $\alpha_Z$ & $\alpha_{CZ}$ &
			$1/s_{C}$ & $1/s_{Z}$ & $1/s_{CZ}$ &
			$\rho$ & $\tau_C$ & $\tau_Z$  \\ 
			\hline
			Independent & 0.69 & 0.35 & 0.77 & 0.90 & - &       94.8 &   188.6 &  - &       - &       0.09 &    0.1        \\ 
			Full &    0.7 &     0.36 &     0.74 &     0.77 &     0.77 &    90.6 &   189.3 &   115.0 &     0.64 &     0.04 &        0.07   \\ 
			Parsimonious & 0.7 &     0.36 &     0.76 &     0.76 &     0.76 &    91.5 &    198.8 &  118.6 &     0.62 &     0.07 &        0.07  \\
			\hline		
		\end{tabular}
	\end{table}
	The maximum likelihood estimates of the full bivariate powered exponential model agree with the independent univariate estimates, see Table \ref{tb:MLEbistable}. The copper and zinc standard deviations are  $\sigma_C = 0.70$ and $\sigma_Z = 0.36$ respectively. There are  nugget effects for copper ($\tau_C = 0.04$) and  for zinc ($\tau_C = 0.07$). The values of the estimated smoothness parameters $\alpha_C = 0.74$ and $\alpha_C = 0.77$ are closer to each other than in the independent model. This is probably due to the positive definiteness restrictions in Theorem \ref{th:exp}, which exclude some parameter combinations with very distinct scale and smoothness parameters and a high correlation, which is estimated as $\rho_{LC} = 0.63$. The estimate of $\rho_{LC}$  agrees well with the colocated empirical correlation.
	
	\begin{figure}
		\includegraphics[width=\maxwidth]{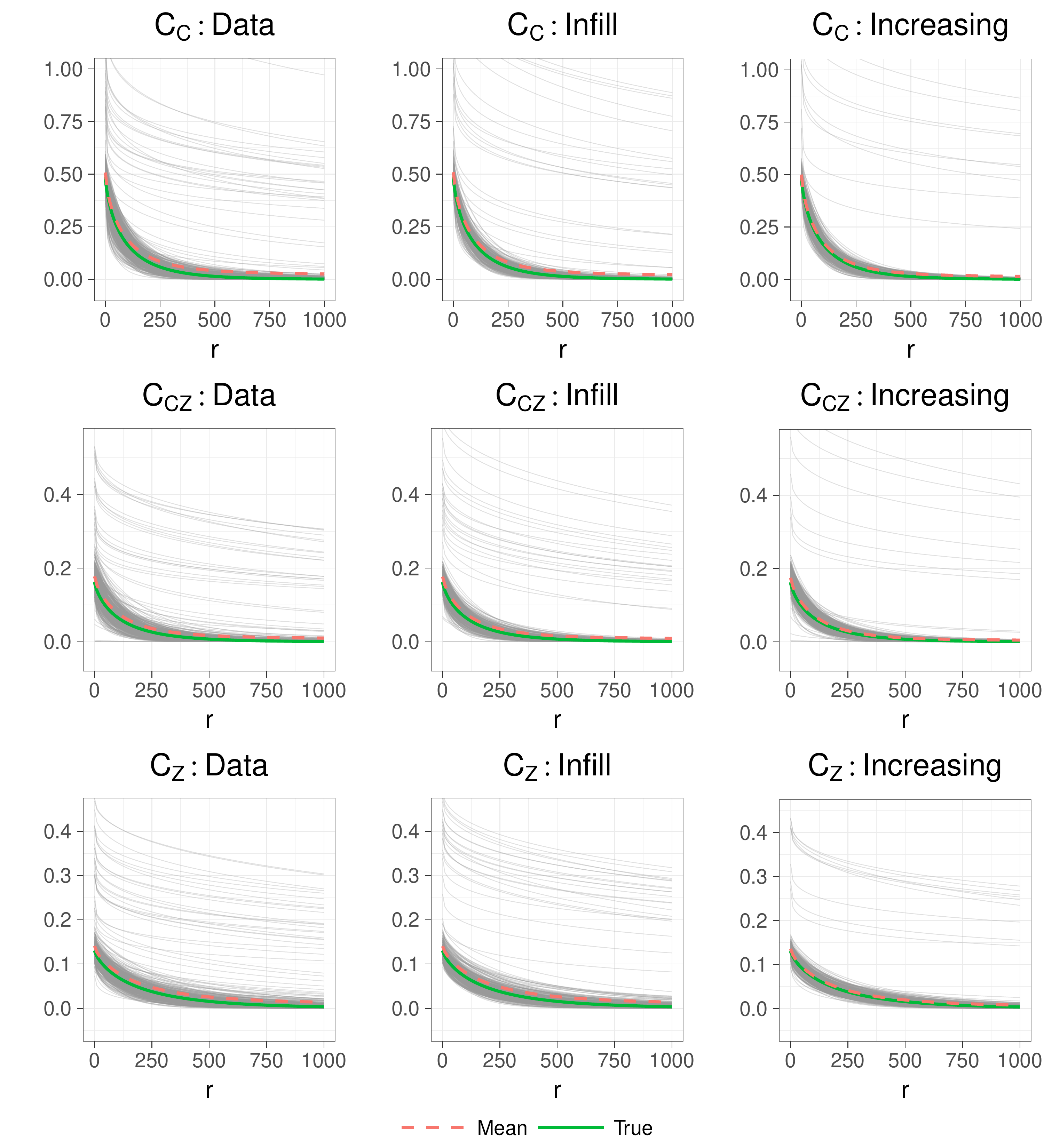} 	
		\caption{Fitted  bivariate powered exponential covariance models for 500 simulated bivariate random fields. The solid thick line is the original covariance model, with which the fields were simulated, the dashed line is the average of 500 fitted bivariate powered exponential models.}\label{fig:covariancesSim}
	\end{figure}

	\begin{figure}
		\centering
		\begin{subfigure}[b]{0.3\textwidth}
			\includegraphics[width=\maxwidth]{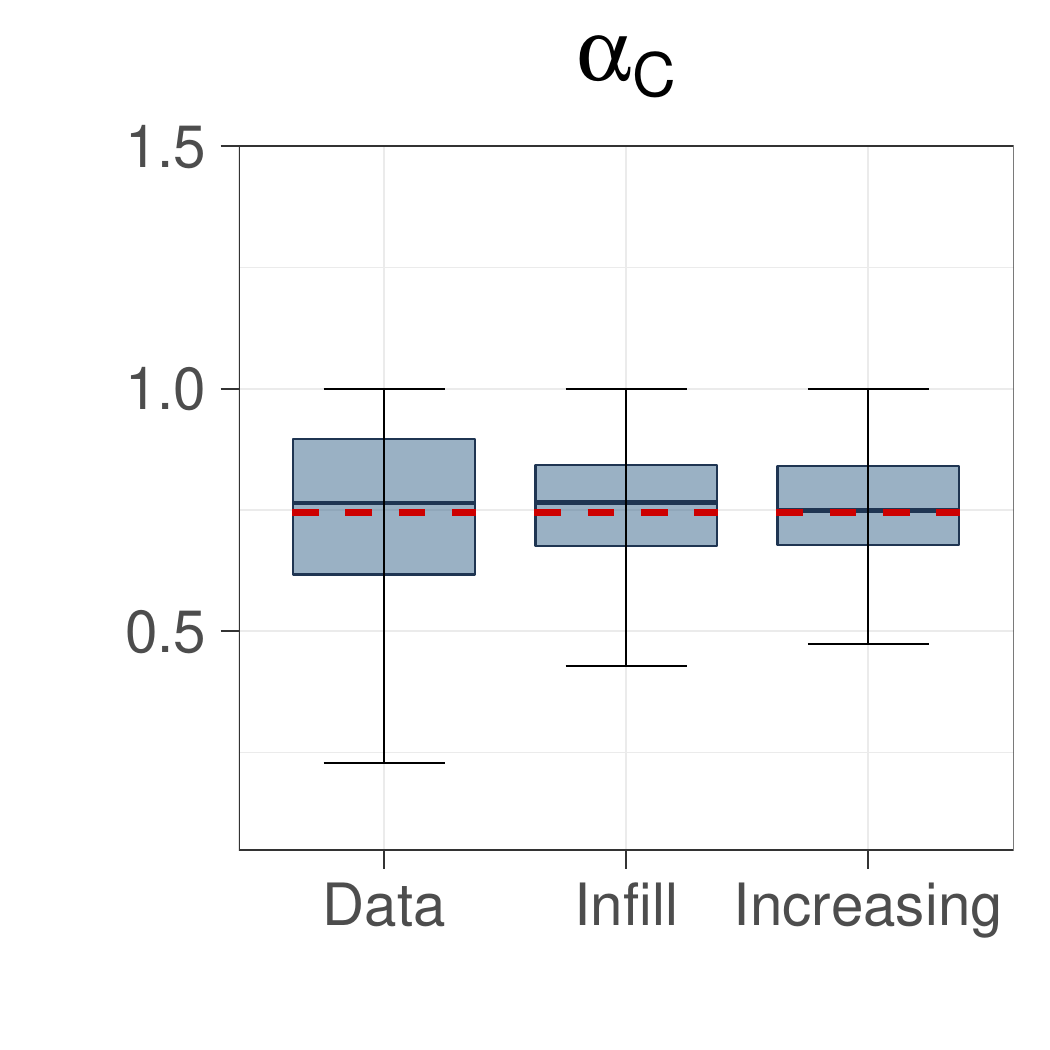} 	
		\end{subfigure}
		\hfil
		\begin{subfigure}[b]{0.3\textwidth}                 
			\includegraphics[width=\maxwidth]{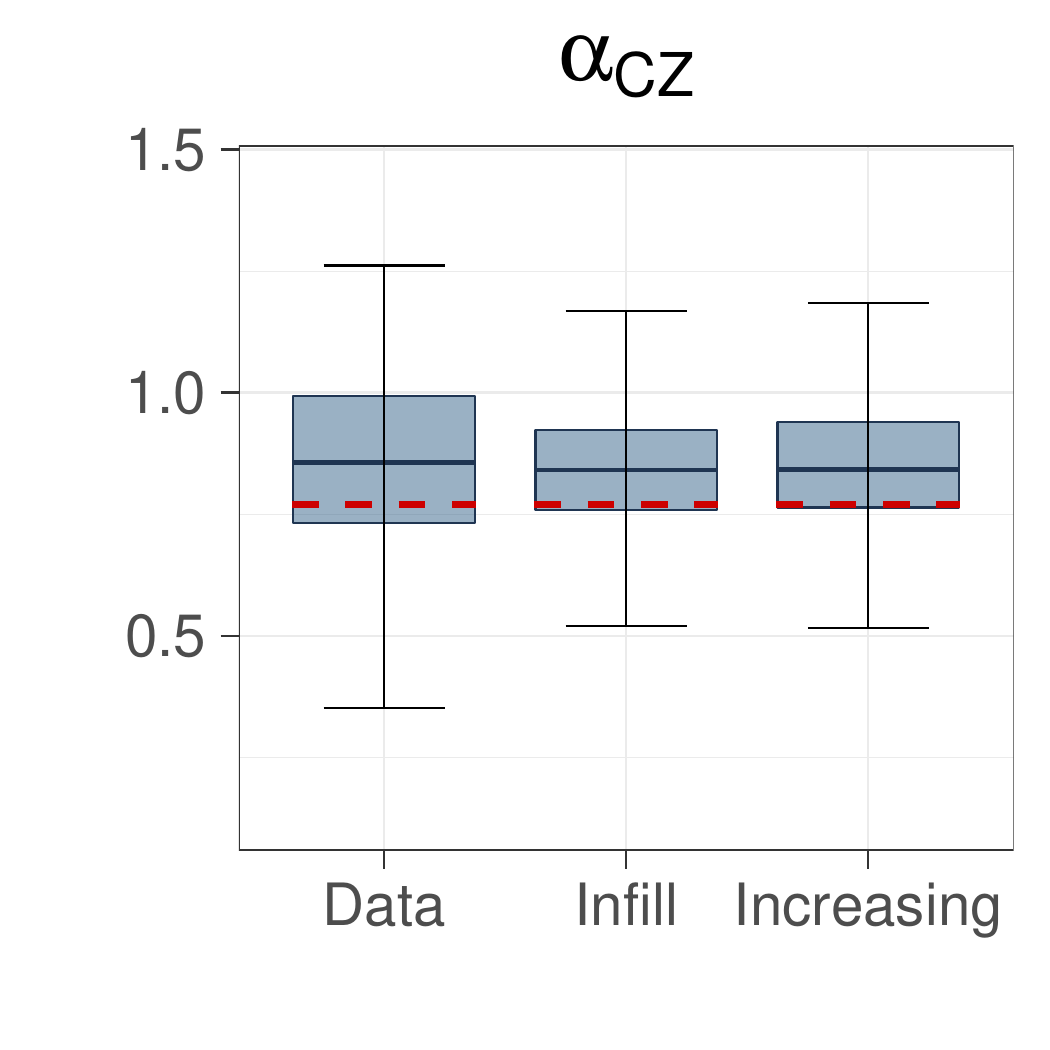} 	
		\end{subfigure}
		\hfil
		\begin{subfigure}[b]{0.3\textwidth}
			\includegraphics[width=\maxwidth]{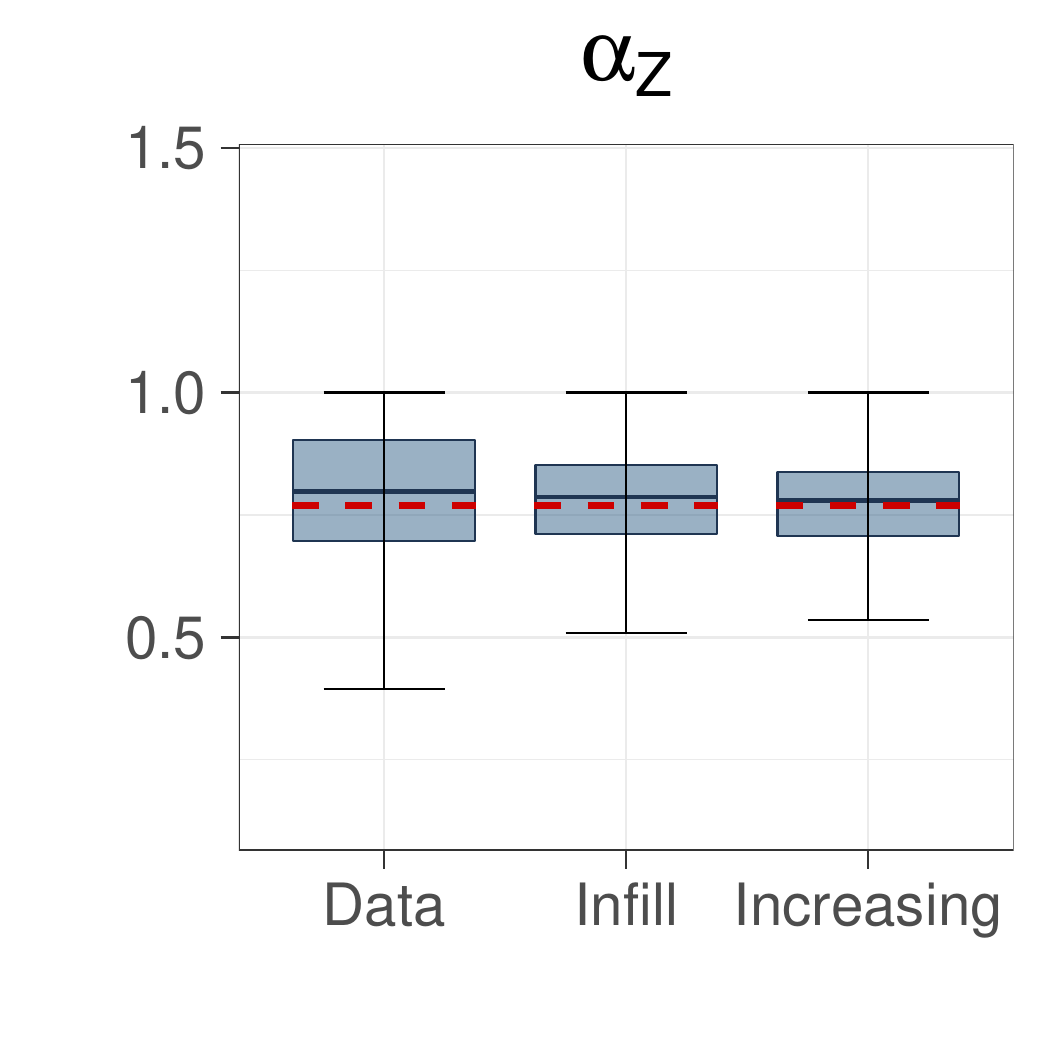} 	
		\end{subfigure}
		\\[3ex]
		\begin{subfigure}[b]{0.3\textwidth}
			\includegraphics[width=\maxwidth]{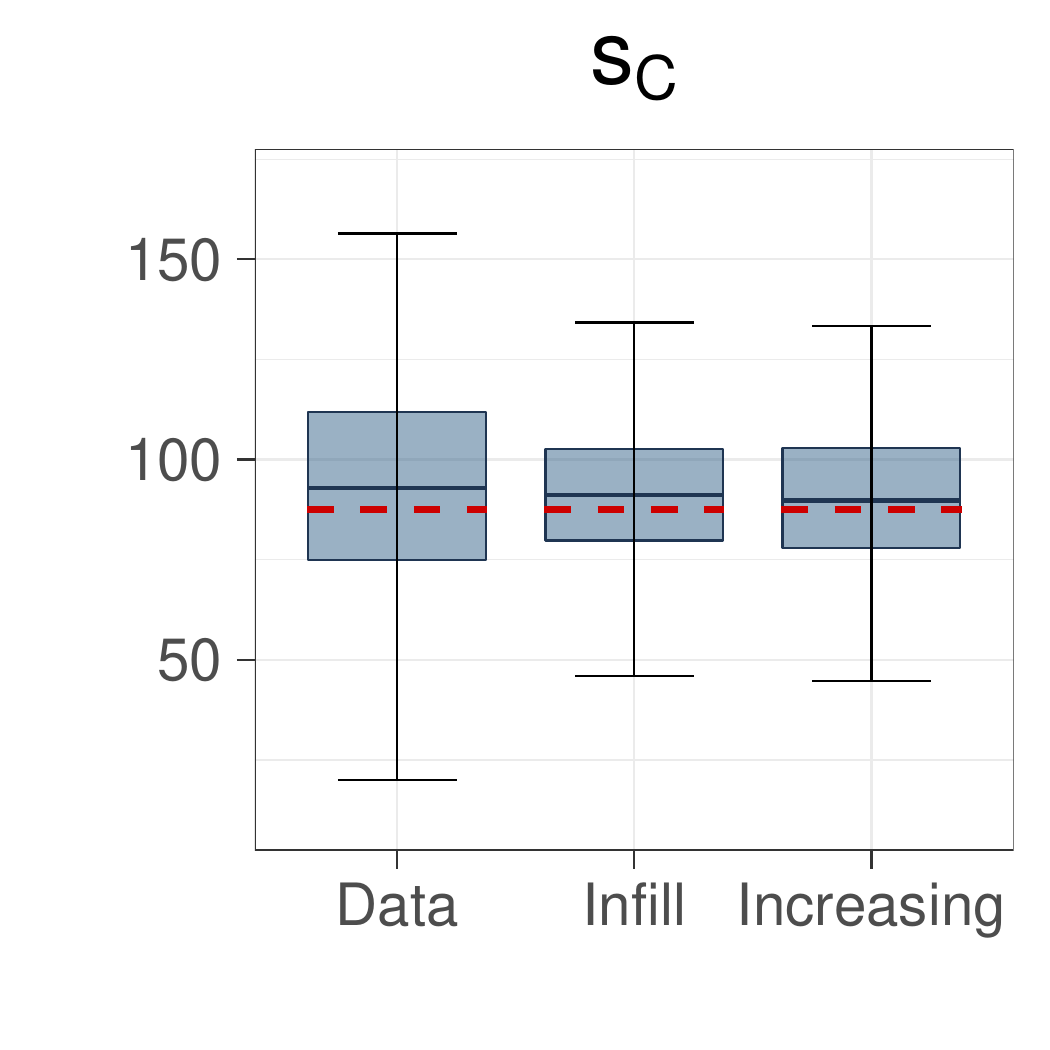} 	
		\end{subfigure}
		\hfil
		\begin{subfigure}[b]{0.3\textwidth}                 
			\includegraphics[width=\maxwidth]{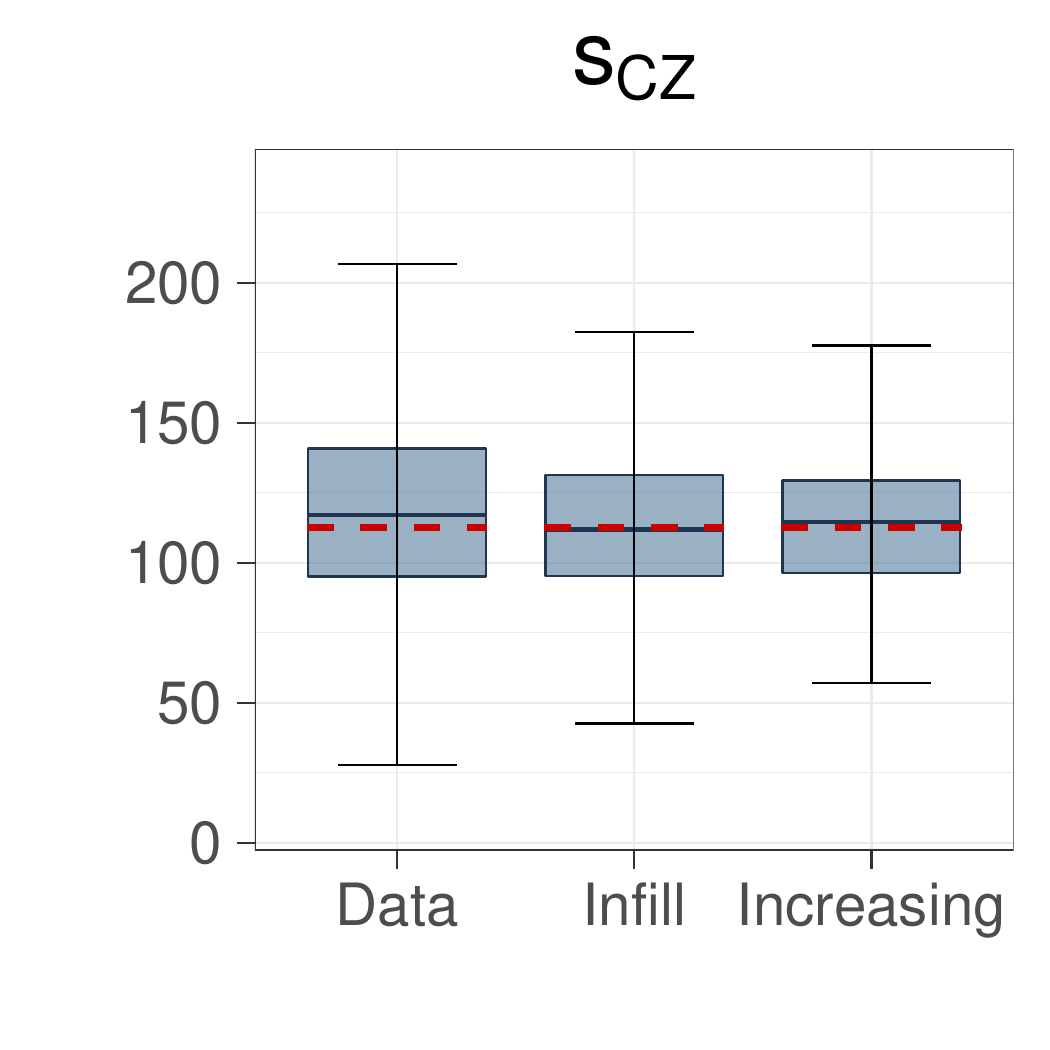} 	
		\end{subfigure}
		\hfil
		\begin{subfigure}[b]{0.3\textwidth}
			\includegraphics[width=\maxwidth]{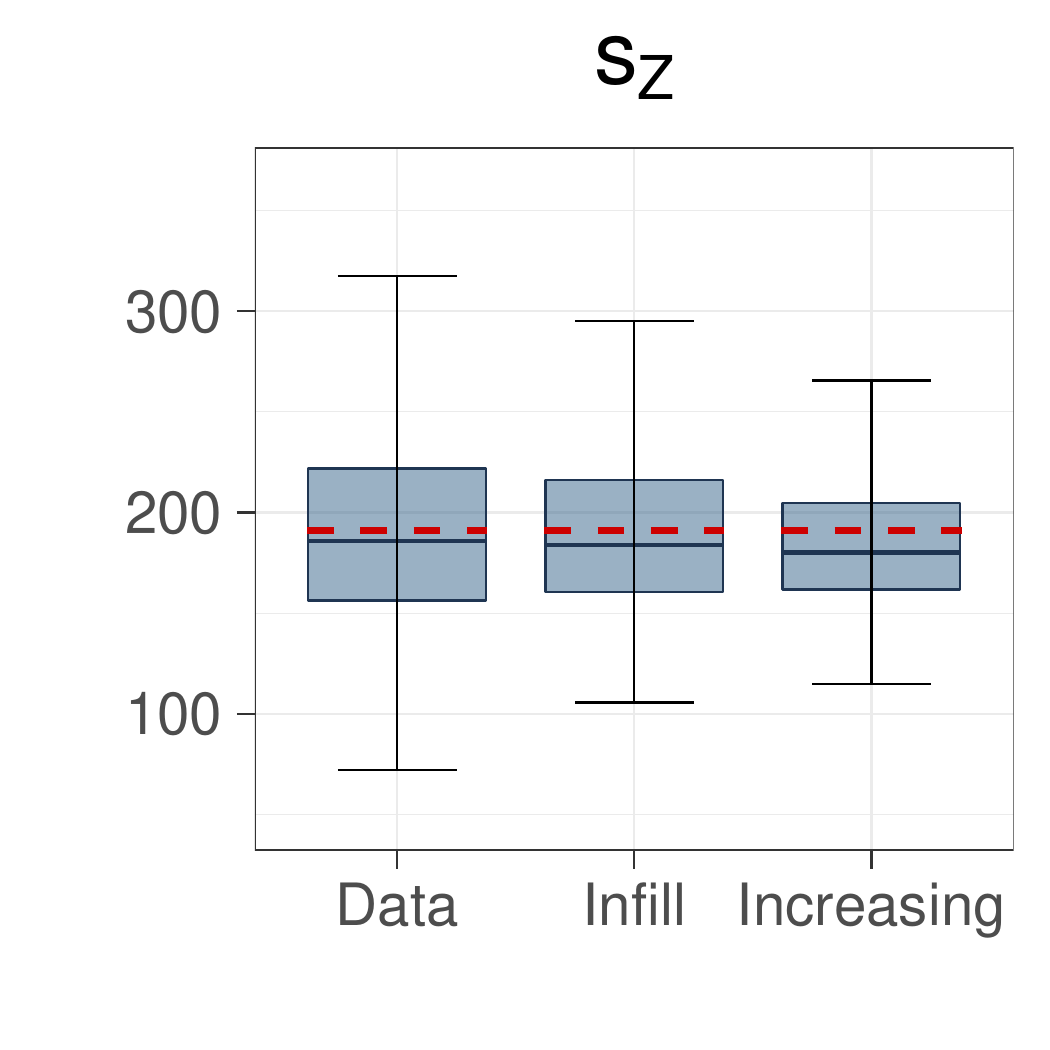} 	
		\end{subfigure}
		\\[3ex]
		\begin{subfigure}[b]{0.3\textwidth}
			\includegraphics[width=\maxwidth]{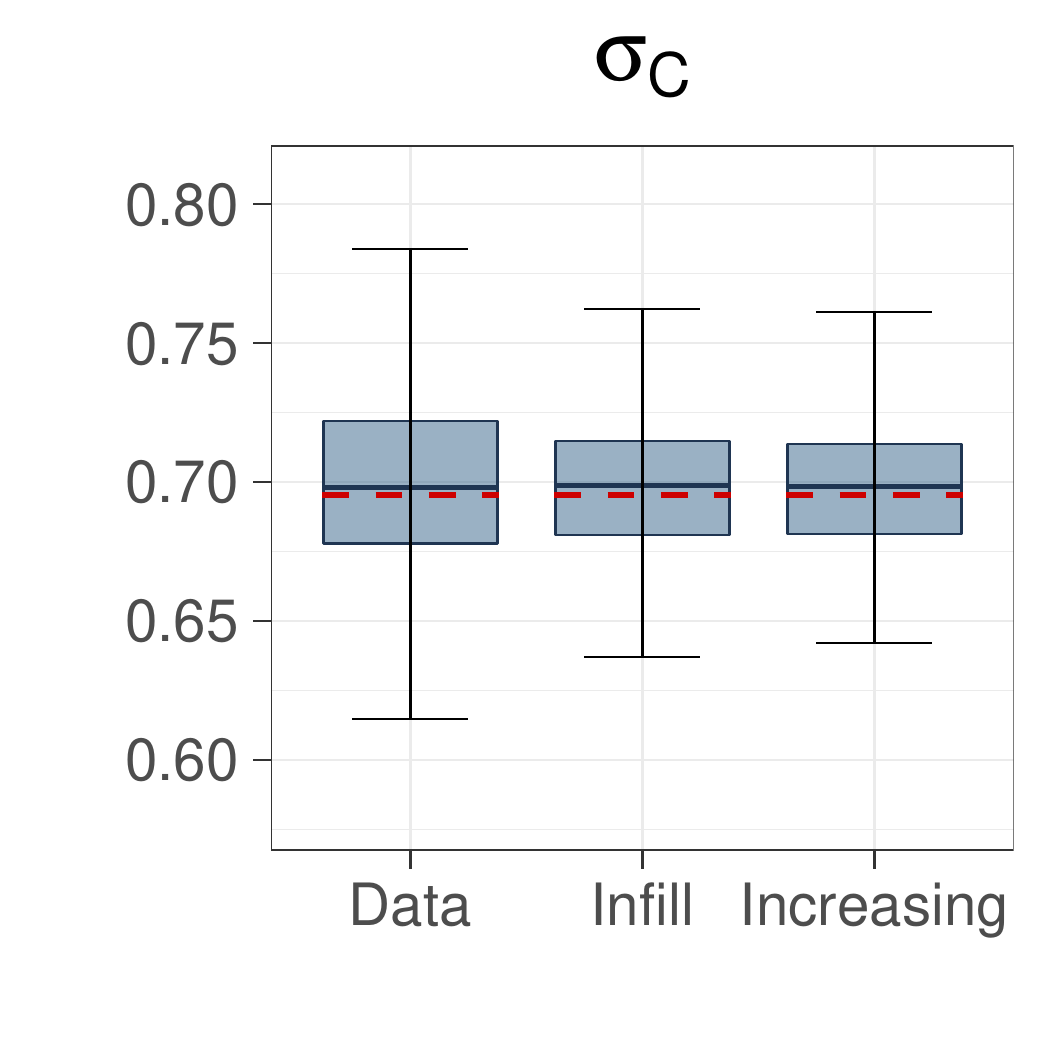} 	
		\end{subfigure}
		\hfil
		\begin{subfigure}[b]{0.3\textwidth}                 
			\includegraphics[width=\maxwidth]{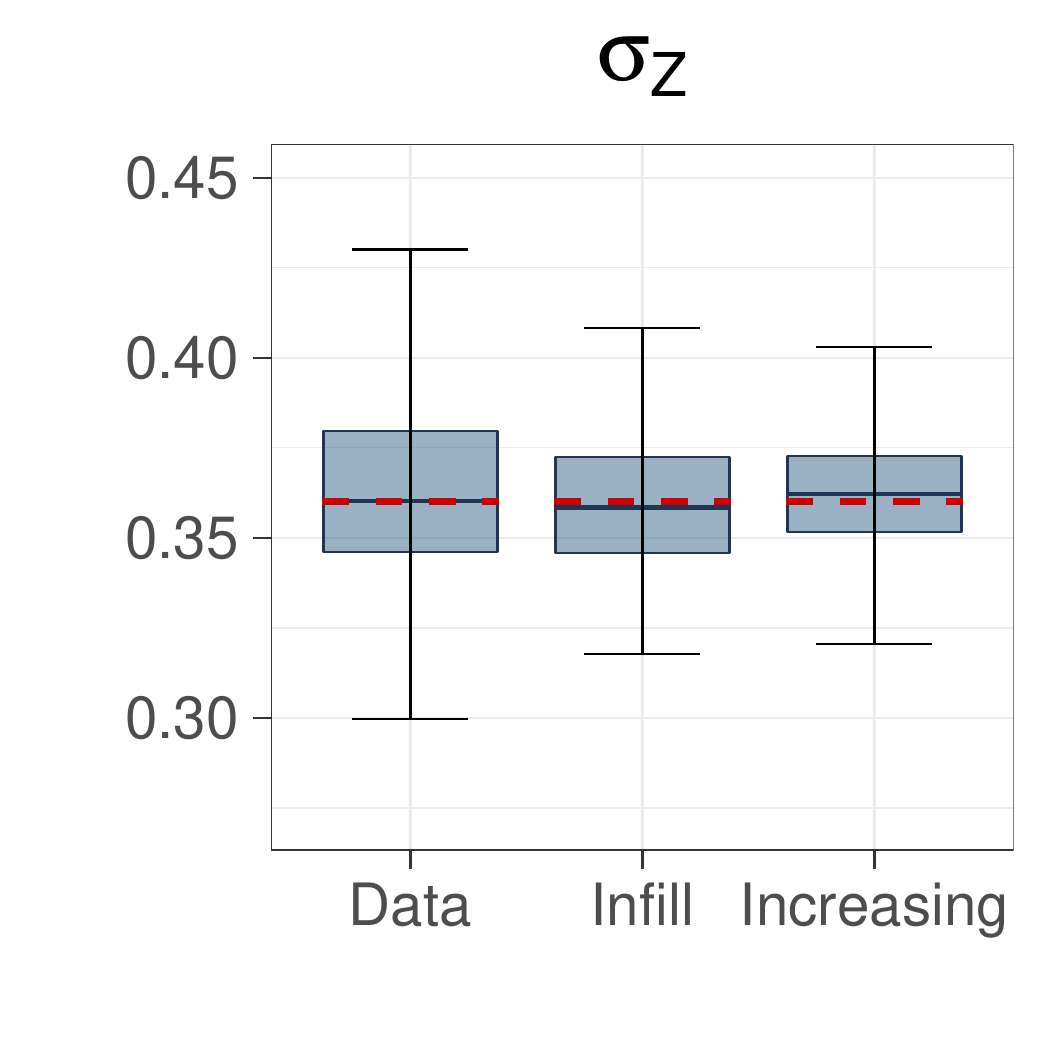} 	
		\end{subfigure}
		\hfil
		\begin{subfigure}[b]{0.3\textwidth}
			\includegraphics[width=\maxwidth]{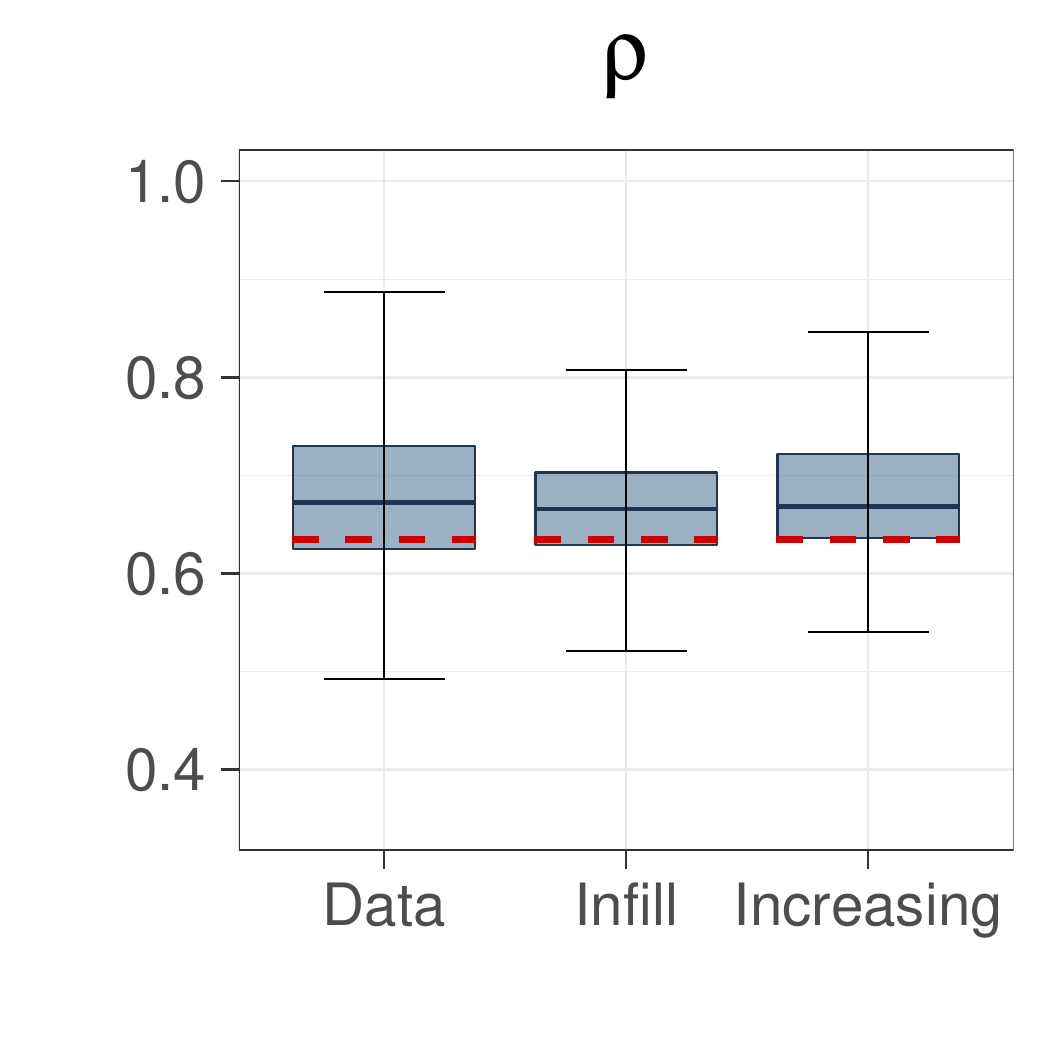} 	
		\end{subfigure}
		\\[3ex]
		\begin{subfigure}[b]{0.3\textwidth}
			\includegraphics[width=\maxwidth]{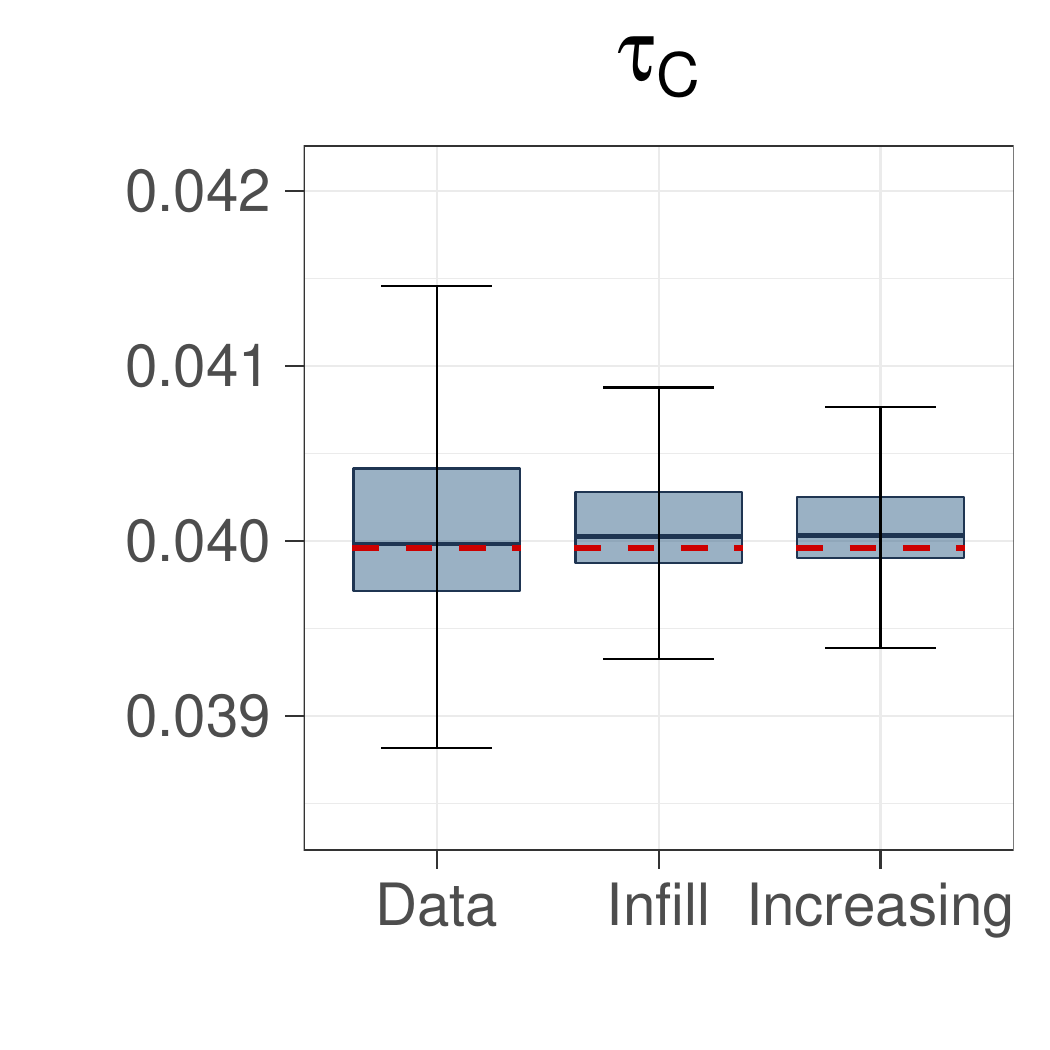} 	
		\end{subfigure}
		\hfil
		\begin{subfigure}[b]{0.3\textwidth}                 
			\includegraphics[width=\maxwidth]{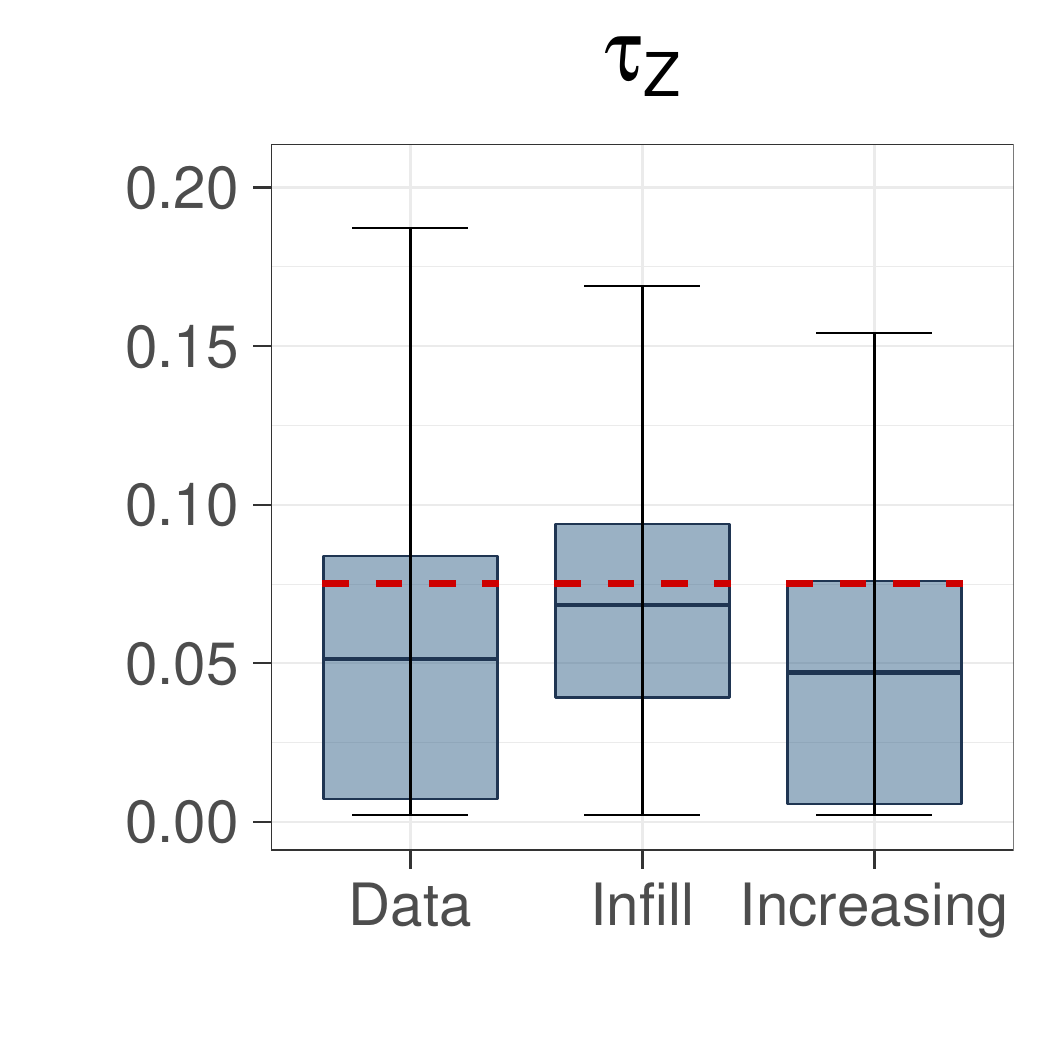} 	
		\end{subfigure}
		\caption{Results of the simulation study for the bivariate powered exponential model, summarized by boxplots of the ML estimates
			for $\sigma_C,$ $\sigma_Z,$ $\rho,$ $\alpha_C,$ $\alpha_{CZ},$ $\alpha_{Z},$ $s_C,$ $s_Z,$ $s_{CZ},$ $\tau_C,$ $\tau_Z$. The boxes range from the
			lower to the upper quartile, and the whiskers extend to the most extreme data point that is no more than 1.5 times the interquartile range 	from the box. The dashed horizontal lines are at the true values. }
		\label{fig:refit}
	\end{figure}

	Following \cite{gneiting2012matern}, in order to assess a typical finite sample variability in the estimation of the bivariate powered exponential model we perform a small simulation study. Specifically, we generate 500 realizations from the full bivariate powered exponential model with parameter values of Table \ref{tb:MLEbistable}. The simulations are done on a 50 by 50 square grid of the area 14.6 km$^2$. For each realization, we choose randomly 259 points of the grid and fit the bivariate powered exponential model by maximum likelihood. The fitted covariance functions are shown in Figure \ref{fig:covariancesSim}. The average of all 500 covariance functions (dashed line) is close to the original model (solid  line). The parameters estimates are summarized by the boxplots in Figure \ref{fig:refit}. The medians of estimates of $\sigma_C$, $\sigma_Z,$ $\rho,$ $s_C,$ $s_Z,$ $s_{CZ}$ are very close to their true values. The interquantile ranges of estimates of $\alpha_C$ and $\alpha_{Z}$ have a large overlapping area with the interquantile range of estimates of $\alpha_{CZ}$. 
	
	Similarly to \cite{gneiting2012matern}, we supplement these finite sample results with a view towards the two common forms of spatial asymptotics, infill and
	increasing domain. For infill asymptotics, we used the same simulation grid, but doubled the number of sample locations (to 518). For increasing domain asymptotics, we increased the domain size in both coordinate directions by a factor of $\sqrt{2}$, while doubling the number of sample locations (to 518), so as to retain
	the original sampling density. 	The factor two was chosen in order to keep the computing time at a reasonable level. Fitted covariance functions  and the boxplots of the corresponding estimates are also included in Figures \ref{fig:covariancesSim} and \ref{fig:refit}, respectively. Generally speaking, parameter estimates are seen to be tighter under both asymptotic frameworks.
	
	Since there is no strong evidence that $\alpha_C$, $\alpha_Z$, $\alpha_{CZ}$ are distinct for the full bivariate powered exponential model, we fit  a parsimonious bivariate powered exponential model with $\alpha_C = \alpha_L = \alpha_{LC} = \alpha$. In addition, we set $\tau_C = \tau_L = \tau$, since the medians of their estimates are close to each other.  Thus, our  parsimonious bivariate powered exponential model becomes
	\begin{align}\label{eq:copperzincbistpars}
	C_{C}(r) & = \sigma_C^2 \exp\left(-(s_C r)^{\alpha}\right) + \tau^2  \mathds{1}(r =0) , \\
	C_{Z}(r) & = \sigma_Z^2 \exp\left(-(s_Z r)^{\alpha}\right) + \tau^2  \mathds{1}(r =0) ,
	\end{align}
	and 
	\begin{equation}\label{eq:copperzincbistparscross}
	C_{CZ}(r)= C_{ZC}(r) = \rho \sigma_C \sigma_Z \exp\left(-(s_{CZ}r)^{\alpha}\right),
	\end{equation}
	where $\alpha \in (0, 1]$,  $s_C, s_Z, s_{CZ} >0$ and $|\rho| <1$ satisfy the conditions of Theorem \ref{th:exp} and $\sigma_C, \sigma_Z, \tau >0$. The parameter estimates of the parsimonious bivariate powered exponential model agree well with those of the full bivariate powered exponential model, see Table~\ref{tb:MLEbistable}. The likelihood of the parsimonious model is only 0.05 smaller than the likelihood of the full model, see Table \ref{tb:comparison}.
	
	Next, we fit  the full bivariate Mat\'{e}rn model, i.e.\
	\begin{align}\label{eq:copperzincbiwm}
	C_{C}(r) & = \sigma_C^2  M_{\nu_C}(\s_{C} r) + \tau_C^2  \mathds{1}(r =0) , \\
	C_{Z}(r) & = \sigma_Z^2  M_{\nu_Z}(\s_{Z} r) + \tau_Z^2  \mathds{1}(r =0) ,
	\end{align}
	and 
	\begin{equation}\label{eq:copperzinccrossbiwm}
	C_{CZ}(r)= C_{ZC}(r) = \rho \sigma_C \sigma_Z  M_{\nu_{CZ}}(\s_{CZ} r) ,
	\end{equation}
	where $\nu_L, \nu_C, \nu_{CZ}, \s_C, \s_Z, \s_{CZ}, \sigma_{C}, \sigma_{Z}, \tau  >0$, $|\rho| \le 1$,  $M_{\nu}(\s r) = 2^{1 - \nu } (\s r)^{\nu} K_{\nu} (\s r)/\Gamma(\nu)$, $K_{\nu} (r)$ is the modified Bessel function of the second kind and $\Gamma$ is the gamma function. The ML estimates  are displayed in Table
	\ref{tb:MLEbiwm}. The estimates of the variance  are close to those in the bivariate powered exponential model, whereas the estimated nugget effects are smaller than those in the bivariate powered exponential model. From the  estimates of the smoothness parameters $\nu_C = 0.3$ and $\nu_{Z} = 0.28$ we get the estimates of the fractal dimensions of copper and zinc fields, which are 2.7 and 2.72 respectively. These values slightly exceed the estimates of fractal dimension in the bivariate powered exponential models (2.63 for copper and 2.62 for zinc in the full model and 2.62 in the parsimonious one).
	\begin{table}
		\centering
		\caption{Maximum likelihood estimates of parameters for the bivariate Mat\'{e}rn model applied to the  copper and zinc data.}
		\label{tb:MLEbiwm}
		\begin{tabular}{llllllllllll}
			\hline
			\hline
			Model & $\sigma_C$ & $\sigma_Z$ &  
			$\nu_C$ & $\nu_Z$ & $\nu_{CZ}$ &
			$1/\s_{C}$ & $1/\s_{Z}$ & $1/\s_{CZ}$ &
			$\rho$ & $\tau_C$ & $\tau_Z$  \\ 
			\hline
			Full  &         0.7 &     0.37 &      0.3 &     0.28 &     0.32 &   155.1 &    337.8 &  185.7 &     0.66 &     0.02 &        0.01    \\ 
			\hline		
		\end{tabular}
	\end{table}
	
	The last model that we fit is the linear model of coregionalization with two latent powered exponential fields. As in the previous cases, we augment the model with nugget effects. We choose two latent fields in order to have a  comparable number of parameters to estimate. The covariance function thus becomes
	\begin{align}\label{eq:copperzinclmc}
	C_{C}(r) & = b_{11}^2 \exp\left(-(s_1 r)^{\alpha_1}\right) + b_{12}^2 \exp\left(-(s_2 r)^{\alpha_2}\right)  + \tau_C^2  \mathds{1}(r =0) , \\
	C_{Z}(r) & = b_{21}^2 \exp\left(-(s_1 r)^{\alpha_1}\right) + b_{22}^2 \exp\left(-(s_2 r)^{\alpha_2}\right)  + \tau_Z^2  \mathds{1}(r =0) ,
	\end{align}
	and 
	\begin{equation}\label{eq:copperzinccrosslmc}
	C_{CZ}(r)= C_{ZC}(r) = b_{11} b_{21} \exp\left(-(\s_1 r)^{\alpha_1}\right) +  b_{12} b_{22} \exp\left(-(\s_2 r)^{\alpha_2}\right)
	\end{equation}
	with $b_{11}, b_{21}, b_{12}, b_{22}, \s_1, \s_2 > 0$, $\alpha_1, \alpha_2 \in (0, 2]$. The ML estimates of the LMC model are displayed in the Table \ref{tb:LMC}. Similarly to the previous models,  the estimated smoothness parameters are close to each other, $\alpha_1 = 0.78$ and $\alpha_2 = 0.79$, whereas the scale parameters are clearly distinct, $1/s_1 = 91.32$, $1/s_2  = 240.04$. The estimated variances, which are given by $\sqrt{b_{11}^2+b_{12}^2} = 0.69$ for copper and by $\sqrt{b_{21}^2+b_{22}^2} = 0.35$ for zinc, agree well with the estimates in the bivariate powered exponential model and the bivariate Mat\'{e}rn model and so do the estimates of nugget effects.
	\begin{table}
		\centering
		\caption{Maximum likelihood estimates of parameters for the LMC model applied to the copper and zinc data.}
		\label{tb:LMC}
		\begin{tabular}{lllllllllll}
			\hline
			\hline
			Model & $b_{11}$ & $b_{12}$ &  
			$b_{21}$ &$ b_{22}$ & 
			$\alpha_{1}$ &	$\alpha_{2}$ &
			$1/\s_{1}$ & $1/s_{2}$ &
			$\tau_L$ & $\tau_C$  \\ 
			\hline
			LMC &         0.68 &       0.1 &      0.18 &      0.31 &      0.78 &      0.79 &     91.32 &    240.04 &       0.1 &         0.07 \\
			\hline		
		\end{tabular}
	\end{table}
	
	Table \ref{tb:comparison} contains the comparison between the bivariate powered exponential, the bivariate Mat\'{e}rn, the independent powered exponential  and the LMC fits. The full bivariate Mat\'{e}rn model achieves the highest likelihood. The parsimonious bivariate powered exponential model has the smallest value of AIC. Having the same number of parameters as the LMC, the parsimonious bivariate powered exponential model  has a higher likelihood value. All bivariate models have higher likelihood and smaller value of AIC than the independent powered exponential model. 
	\begin{figure}
		\includegraphics[width=\maxwidth]{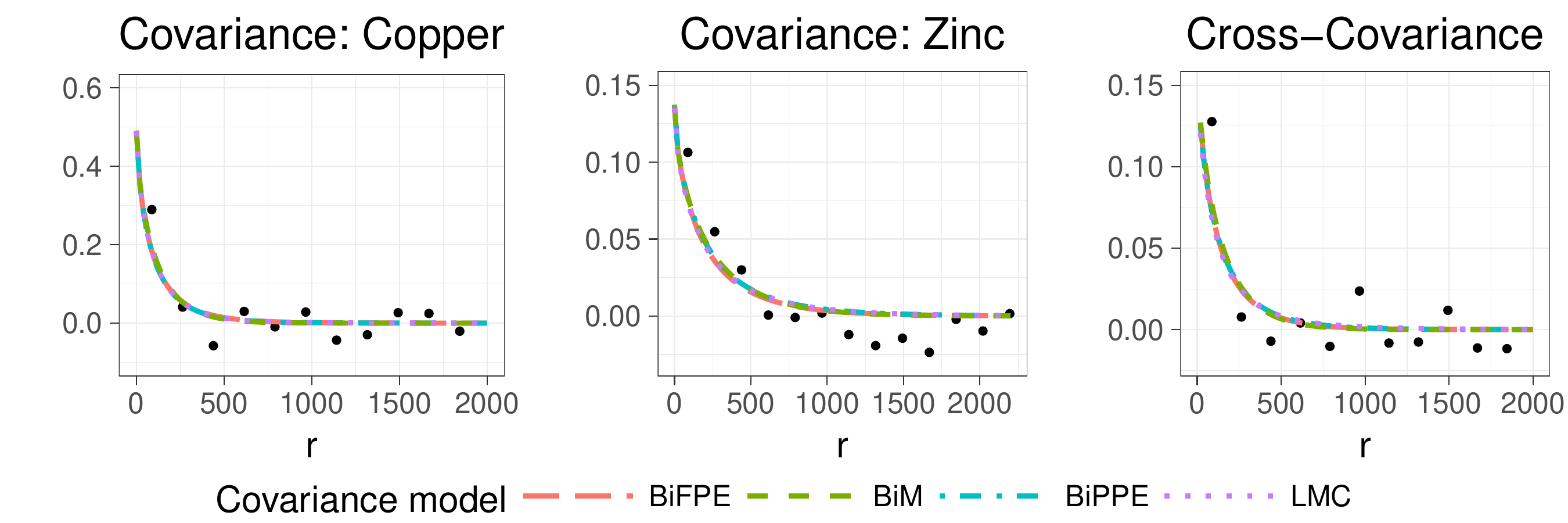} 	
		\caption{Empirical covariance and bivariate covariance functions for the copper and zinc data, with maximum likelihood fits under
			the full bivariate powered exponential model, (BiFPE; long dashed line),  
			the bivariate Mat\'{e}rn  (BiW; green dashed line),
			the parsimonious bivariate powered exponential model, (BiPPE;  dashed dotted line),  
			and the linear model of coregionalization (LMC; dotted line).}\label{fig:cov}
	\end{figure}
	
	We compare predictive performance of the models on the validation set. First, we take the logarithm of copper and zinc in the test set and then subtract the mean of logarithms of copper and zinc, respectively, from the training set.  At the test set locations we perform co-kriging  to predict the values for copper and zinc. Then we calculate  the mean absolute error (MAE), i.e.\ the average absolute error between the realization and the co-kriging point predictor. When we do not use the measurements of zinc from the test set for copper prediction, there is no gain in exploiting the bivariate models. The same holds for the zinc prediction without using copper values. Smaller MAE is achieved when the measurements of zinc concentrations are included for copper prediction and vice versa. The results are summarized in Table \ref{tb:comparison}. The bivariate models clearly outperform the independent model both in copper and zinc.

	\begin{table}
		\centering
		\caption{Comparison of the bivariate powered exponential, the bivariate Mat\'{e}rn, the independent powered exponential  and the LMC models for copper and zinc data.}
		\label{tb:comparison}
		\begin{tabular}{lccccc}
			\hline
			\hline
			Model & \thead{Number of \\ parameters}   & Log likelihood & AIC & \thead{MAE \\ (copper)} & \thead{ MAE \\ (zinc)} \\  
			\hline
			Full bivariate powered exponential & 11 & -181.42   & 384.84 &  0.5544 &  0.2316  \\
			Parsimonious powered exponential & 8 & -181.47   & 378.93 & 0.5551 &  0.2320  \\
			Full bivariate Mat\'{e}rn & 11 & -181.21   & 384.42   &  0.5593 &   0.2347  \\
			LMC & 10 &  -181.59 & 383.19 &  0.5534 &  0.2292  \\
			Independent powered exponential & 8 &  -245.6   & 507.22 &0.5764 &  0.2742  \\
			\hline		
		\end{tabular}
	\end{table}

\section{Acknowledgments}
	The authors gratefully acknowledge support by Deutsche Forschungsgemeinschaft through the Research Training Group RTG 1953. OM acknowledges support by Mannheim University through the dissertation completion grant. The authors are grateful to Tilmann Gneiting for suggestions and valuable comments on the paper.

\begin{center}
		\appendix
\end{center}

{\centering\section{Bivariate spherical model}\label{app:sph}}

	\begin{lemma}\label{lemma:schoenberg}
	Let $[f_{ij}]_{i,j = 1}^m,$  be the spectral density matrix of  an $m$-variate continuous covariance function $\textbf{C}$. Then the set of roots of $f_{ij}$ is a superset of the roots of $f_{ii}$ and the roots of $f_{jj}$ for any $i,j = 1, \dots, m,$ $i \neq  j$.
	\end{lemma}
	\begin{proof}
	 The lemma follows directly from Schoenberg's theorem.
	\end{proof}	
		
	\begin{theorem}\label{thm:sequence}
	Let $[f_{ij}]_{i,j = 1}^m,$  be the spectral density matrix of a stationary and isotropic covariance function~$\textbf{C}$, $i, j = 1, \dots, m$. Suppose that  there exists a positive strictly increasing sequence $(u_k)_{k \in \N}$ such that the following properties hold:
	\begin{enumerate}[label=(\roman*)]
		\item for any $s < 1$, there exists a $k_0 \in \N$ with $u_{k_0}/s \neq u_k$ for all $k \in \N,$ \label{item:sequences}
		\item for some $s_{ij} > 0$ the elements of the sequence $(s_{ij} u_k)_{k \in \N}$  constitute all roots of $f_{ij}$, $i, j = 1, \dots, m$. \label{item:sequencef}
	\end{enumerate}
	Then $\s_{ij} = s$ for some $s > 0$ and all $i, j = 1, \dots, m$. 	
	\end{theorem}
	\begin{proof} We prove the theorem for $m = 2$. The proof for $m > 2$ follows immediately from the properties of positive definite matrices and their determinants.

We denote by $A_{ij} = \{s_{ij} u_k, \ k \in \N\},$  the set of roots of the function $f_{ij},$ $i,j = 1, 2$. For any $i \in \{1, 2\}$ we have:

		\begin{itemize}
			\item if $s_{12} > s_{ii}$, then  $s_{ii} u_1 \notin A_{12} $ and by Lemma \ref{lemma:schoenberg}, the function $\textbf{C}$  cannot be positive definite.
			\item $ s_{12} < s_{ii}$, then by condition \ref{item:sequences} there exists a $k_0$ such that 
			$\frac{s_{ii}}{s_{12}} u_{k_0} \neq u_k$ for all $k \in \N$ and therefore $s_{ii}u_{k_0} \notin A_{12}$. Again, by Lemma~\ref{lemma:schoenberg}, the function $\textbf{C}$  cannot be positive definite.
		\end{itemize}
\end{proof}

\begin{lemma}\label{lemma:tan}
	Let  $(u_k)_{k \in \N}$ be a sequence such that  $u_k - a k \uparrow b$ strictly monotonically for some $a > b > 0$ as $k$ tends to infinity. Then for any $s < 1$, there exists a $k_0 \in \N$ such that $u_{k_0}/s \neq u_k$ for all $k \in \N$.
\end{lemma}
\begin{proof}
	We prove the lemma by contradiction. Note that $u_k$ is strictly monotone for $k \ge k_0$ and some $k_0  \in \N$. Suppose that there exists an $s <1$ such that for all $k \in \N$  there exist $l_{k} \in \N$ with $u_{k}/s =  u_{l_{k}}$. 
	First note that there exists an $N \in \N$ such that for every $k \ge N$ the corresponding $u_k$ lies inside the interval $\left(b + a (k-1), b + a k \right)$ and there exists a decreasing sequence $\varepsilon_k \downarrow 0,$ $\varepsilon_k  \in (0, a),$ such that 
	\begin{equation*}
	u_k = b + a k  -  \varepsilon_k.
	\end{equation*}
	For any $s < 1$, there exists an $n \in \N$ and $0 \le c < 1$ such that  $a/s = a n - a c$.  We consider $k \ge k_0$. Then $l_k > k.$ 
	
	Consider the following cases.
	\begin{enumerate}[label=(\roman*)]
		\item $c = 0$. There exist $n_b \in \N$ and $c_b \in [0, 1)$ such that $b/s = b + a(n_b - c_b)$. Then we have 
		\begin{align*}
		u_{l_k} = \frac{u_{k}}{s}  = \frac{b + ak - \varepsilon_k   }{s} 
		= b + a(n_b - c_b) + ank - \frac{\varepsilon_k   }{s}  
		= b + a(n_b + nk)  - \left(ac_b+ \frac{\varepsilon_k   }{s} \right).
		\end{align*}
		We choose  $k$	large enough so that $0 < ac_b+ \frac{\varepsilon_k   }{s} < a$. 	Since 
		$$b +  a(n_b + kn - 1)  < u_{l_k} <  b +  a(n_b + kn),$$ 
		we get $l_k = n_b + kn$ and $\varepsilon_{n_b + kn} = ac_b+ \varepsilon_k /s$.  But then it follows that  $\varepsilon_{l_k} > \varepsilon_k$, which cannot be true, since $(\varepsilon_k)_{k \in \N}$ is a decreasing sequence.
		
		\item $c > 0$.	By our assumption, for $u_{k+1}$  there exists $l_{k+1} > k+1$, such that $\frac{u_{k+1}}{s} = u_{l_{k+1}}$.	
		We obtain
		\begin{align*}
		\frac{u_{k+1}}{s} & = \frac{u_k + a -(\varepsilon_{k+1} - \varepsilon_k )  }{s} \\
		&= u_{l_k} + a n -  a c - \frac{\varepsilon_{k+1} - \varepsilon_k}{s} \\
		& = b + a l_k - \varepsilon_{l_k} + a n -  a c - \frac{\varepsilon_{k+1} - \varepsilon_k}{s} \\
		& = b + a (l_k + n)  -  \left(a c +  \varepsilon_{l_k} + \frac{\varepsilon_{k+1} - \varepsilon_k}{s}\right) 
		\end{align*}
		Choose $k$ large enough, so that $0 < a c +  \varepsilon_{l_k} + \frac{\varepsilon_{k+1} - \varepsilon_k}{s} < a $. Then $l_{k+1} = l_k + n$ and $\varepsilon_{l_k + n} = a c +  \varepsilon_{l_k} + \frac{\varepsilon_{k+1} - \varepsilon_k}{s}$. Note that $l_k + n \to \infty$ and  $\varepsilon_{l_k + n}  \to ac$ when $k \to \infty$, which is a contradiction, since $c > 0$. 	
	\end{enumerate} 		
\end{proof}

\begin{corollary}
The bivariate spherical model \eqref{eq:sph} is a valid covariance model in $\R^3$ if and only if $\rho = 0$ or $\s_{11} = \s_{12} = \s_{22}$. 	
\end{corollary}
\begin{proof}
		The spectral density of the univariate spherical correlation functions is 
	\begin{align*}
	f(u) =  \frac{3 s}{\pi^2 u^6} (u \cos(u/2s) - 2s \sin (u/2s)  )^2  	
	\end{align*}
	Clearly, $f$ is pseudo periodic and takes infinitely many zeros on $u > 0$. We denote by $u_k$, $k \in \N,$ the roots of the function $\tilde{f}(u) = u  - \tan(u)$ on $u > 0$. Then the roots of the spectral density $f_{ij}$ are $2s_{ij} u_k$, $k \in \N$, $i, j = 1,2$. Since $u_k \uparrow \frac{\pi}{2} + \pi k$ as $k \to \infty$, Lemma \ref{lemma:tan} and Theorem \ref{thm:sequence} prove the corollary.
\end{proof}

{\centering\section{Sufficient conditions for positive definiteness}\label{app:suff}}
	
	\cite{porcu2011} provide the following construction principle for multivariate covariance models.  
	\begin{theorem}\label{th:porcu}
		\begin{enumerate}[label=\Alph*.]
			\item Let $(\Omega, \mathcal{F}, \mu)$ be a measure space and $E$ be a linear space. Assume that the family of matrix-valued functions $A(x, u) = [A_{ij}(x, u)] : E \times \Omega \mapsto \R^{m \times m}$ satisfies the following conditions: \label{item:porcua}
			\begin{enumerate}[label=(\alph*)]
				\item for every $i, j = 1, \dots, m$ and $x \in E$, the functions $A_{ij}(x, \cdot )$ belong to $L_1(\Omega, \mathcal{F}, \mu)$; \label{item:porcua1}
				\item $A(\cdot, u)$ is a positive definite matrix-valued function for $\mu$-almost every $u \in \Omega$. \label{item:porcua2}
			\end{enumerate}
			Let 
			$$
			C(x) := \int_{\Omega} A(x, u)d\mu(u) = \left[\int_{\Omega} A_{ij}(x, u)d\mu(u) \right]_{i,j=1}^m, \qquad x \in E.
			$$
			Then $\C$ is a positive definite matrix-valued function in $E$.
			\item Conditions (a) and (b) in part \ref{item:porcua} are satisfied when $A(x, u) = k(x, u)g(x, u)$, where the maps $k(x, u): E \times \Omega \mapsto \R$ and $g(x, u) = \left[g_{ij}(x, u)\right]_{i,j=1}^m : E \times \Omega \mapsto \R^{m \times m}$ satisfy the following conditions:  \label{item:porcuB}
			\begin{enumerate}
				\item for every $i, j = 1, \dots, m$ and $x \in E$, the functions $k(x, \cdot)g_{ij}(x, \cdot)$ belong to  $L_1(\Omega, \mathcal{F}, \mu)$;\label{item:porcu1}
				\item $k( \cdot, u)$ is positive definite for  $\mu$-almost every $u \in \Omega$;\label{item:porcu2}
				\item $g(\cdot, u)$ is a positive definite matrix-valued function  or $g(\cdot, u) = g(u)$ is a positive definite matrix for $\mu$-almost every $u \in \Omega$.\label{item:porcu3}
			\end{enumerate}
		\end{enumerate}
	\end{theorem}
	
	Starting from known functions $k$ and $g_{ij}$, \cite{porcu2013radial} and \cite{daley2015}, see also \cite{randomFields}, construct new compactly supported multivariate covariance functions. Our approach, inspired by  \cite{gneiting1999radial}, is different; we consider the model \eqref{eq:class} as a candidate for a multivariate covariance function and then find the corresponding $g_{ij}$, which depend on parameters $s_{ij}$, $\bm{\theta_{ij}}$, and the parameter set which guarantees its positive definiteness.

	\begin{proof}[Proof of Theorem \ref{thm:scale}]
	In Theorem \ref{th:porcu} B we take  a Euclidean space $\R^n,$ $n \in \{1, 3\},$ as $E$ and the Lebesgue measure as $\mu$. 	We first prove the assertion in $\R$. We take  $k(r, u) = \left(1 - \frac{r}{u} \right)_+$, $g_{ii}(u) =  \sigma_i^2 u \psi_{ii}''(u)$, $i = 1, 2,$ and $g_{12}(u) = g_{21}(u) = \rho \sigma_1 \sigma_2 u \psi_{12}''(u)$ for $r \ge 0$, $u > 0$ and such that $\psi''_{ij}(u)$ are defined. We check the conditions of Theorem \ref{th:porcu} B consequently. Conditions \ref{item:firstderiv} and \ref{item:firstderiv2} allows us to apply integration by parts in the following integral, see for example Chapter 10.13 in \cite{apostol1974mathematical}, 
	\begin{align}\label{eq:integrability}
	\int_{0}^{\infty} u \psi_{ij}''(u) du & = u \psi_{ij}'(u) |_0^{\infty} - \int_{0}^{\infty}  \psi_{ij}'(u) du   = \psi_{ij}(0) < \infty.
	\end{align}
	From  equation \eqref{eq:integrability} follows the  condition  \ref{item:porcu1} in Theorem \ref{th:porcu}. Clearly, $k(\cdot, u)$ is a positive definite function in $\R$ for $u >0$ and therefore \ref{item:porcu2} in Theorem \ref{th:porcu}  holds.  Condition \ref{item:porcu3} in Theorem \ref{th:porcu}  is satisfied due to condition \ref{item:pd}. Then the following matrix-valued function is positive definite 	
	\begin{equation}\label{eq:gmatrix}
	\begin{bmatrix}
	\sigma_{1}^2 \int_{0}^{\infty} \left(1 - \frac{r}{u} \right)_+ u \psi''_{11} (u)du & 
	\rho \sigma_{1}\sigma_{2} \int_{0}^{\infty} \left(1 - \frac{r}{u} \right)_+ u \psi_{12}''(u) du \\
	\rho \sigma_{1}\sigma_{2} \int_{0}^{\infty} \left(1 - \frac{r}{u} \right)_+  u \psi_{12}''(u) du &  
	\sigma_{2}^2 \int_{0}^{\infty} \left(1 - \frac{r}{u} \right)_+ u \psi''_{22}(u) du
	\end{bmatrix}.
	\end{equation}
	To simplify the function \eqref{eq:gmatrix} we apply integration by parts again. For $r \ge 0$ we have 
	\begin{align}
	\int_{0}^{\infty} \left(1 - \frac{r}{u} \right)_+ u \psi_{ij}''(u)du  
	& = \int_{r}^{\infty} (u - r) \psi_{ij}''(u)  du \label{eq:integralcont} \\
	& = \int_{r}^{\infty} (u - r) d \psi_{ij}'(u)  \notag \\
	& =  \left. (u - r) \psi_{ij}'(u) \right|_r^{\infty} 
	- \int_{r}^{\infty}\psi_{ij}'(u)   d u \notag \\
	& = \psi_{ij}(r). \notag 
	\end{align}
	Thus, \eqref{eq:gmatrix} and \eqref{eq:class} are the same matrices. 
	
	The proof for $\R^3$ is analogous with $k(r, u) = \left(1 - \frac{r}{u} \right)_+ - \frac{r}{2u}\left(1 -  \frac{ r^2}{u^2}\right)_+$ and $g_{ii}(u) = \frac{1}{3}\sigma_i^2  ( u \psi_{ij}''(u) - u^2 \psi_{ij}'''(u)),$ $i= 1, 2$, $g_{12}(u) =  \frac{1}{3} \rho \sigma_1 \sigma_2 ( u \psi_{12}''(u) - u^2 \psi_{12}'''(u)),$ $r \ge 0$, $u > 0$ and such that $\psi_{ij}(u)''',$ $i, j = 1,2,$ are defined.
	\end{proof}

	The functions $k(r, u)$  are equal to Euclid's hat function, $k(r, u) = h_n(r/u), n = 1, 3$ \citep{gneiting1999radial}. Thus, Theorem \ref{thm:scale} can be generalized to higher dimensions with corresponding functions $h_n$, but it requires the calculation of higher order derivatives. The generalization of Theorem \ref{thm:scale} for processes with more than two components is straightforward. 	Theorem \ref{thm:scale} can be seen as a generalization of the criteria of P\'{o}lya type for radial positive
	definite functions in $\R$ and $\R^3$ (cf. Gneiting (2001); Gneiting et al. (2006)) for bivariate 	fields.
	Condition \ref{item:firstderiv} in $\R$ and $\R^3$ is  not restrictive and fulfilled by many model classes, including the Mat\'{e}rn model.
\begin{proof}[Proof of Theorem \ref{th:exp}]
	Functions $\psi_{ij}(r|\alpha_{ij}, s_{ij})$, $i, j = 1, 2$ of the bivariate powered exponential model satisfy the requirements of Theorem \ref{thm:scale}. Inequality \eqref{eq:exprho} follows directly from the inequalities \eqref{eq:inf}  and \eqref{eq:inf2}.
	All factors of the right-hand side of  inequality  \eqref{eq:exprho} are positive for $r > 0$. That means that the infimum can be zero only at    $r = 0$ or $r = \infty$. Clearly, for the parameters values given in  $(i) - (iv) $, the infimum is positive  and it is zero for other parameter values. 
	Consider now the case  $\alpha_{12} < (\alpha_{11} + \alpha_{22})/2$. Note that for $\alpha \in (0, 1)$ the spectral density $f$ of   $\psi(r) = \exp\left(-r^{\alpha}\right)$, $ r > 0,$ decays at infinity as	
	\begin{equation*}
	f(u) \aspropto   u^{-\alpha-n}  \text{ as } u \to \infty, 
	\end{equation*}	
	This follows from Tauberian theorem \citep{bingham1972tauberian} and Remark 35 in Chapter 2 of \cite{yaglom19872}.  
	Then by Scoenberg's theorem, the bivariate powered exponential model requires necessarily $\alpha_{12} \ge (\alpha_{11} + \alpha_{22})/2$ unless $\rho = 0$.

\end{proof}	

\begin{proof}[Proof of Theorem \ref{th:cauchy}]
	Functions $\psi_{ij}(r|\alpha_{ij}, \beta_{ij}, s_{ij})$, $i, j = 1, 2,$ of the bivariate generalized Cauchy model satisfy the requirements of Theorem \ref{thm:scale}. Inequality~\eqref{eq:cauchyrho} follows from inequalities \eqref{eq:inf}  and \eqref{eq:inf2}.  Analogously to the bivariate powered exponential model, all factors of the right-hand side of  inequality  \eqref{eq:cauchyrho} are positive for $r > 0$. That means that the infimum can be zero only at    $r = 0$ or $r = \infty$. Clearly, for the parameter values given in  \ref{item:cauchy6} the infimum is positive and it reaches zero at infinity for the parameter values 
	in  \ref{item:cauchy5}. The cases \ref{item:cauchy1} and \ref{item:cauchy2} follow from the Schoenberg's theorem and the asymptotics of the generalized Cauchy spectral density, see \cite{lim2009gaussian}. 
\end{proof}	

The restriction $\alpha_{12} < (\alpha_{11} + \alpha_{22})/2$ necessarily leads to the  independence of the components in both the bivariate powered exponential model and the bivariate generalized Cauchy model. The same restriction is imposed on smoothness parameters in the full bivariate  Mat\'{e}rn model \citep{gneiting2012matern} and is common for all models of the type \eqref{eq:class}. It stems from the asymptotic behaviour of the spectral density at infinity. Similar condition caused by the asymptotic behaviour of the spectral density at zero 
is imposed on the long range parameters $\beta_{ij},$ $i, j = 1, 2,$ in the bivariate generalized Cauchy model.  For functions with unknown spectral densities Tauberian theorems can be used to determine the asymptotic behaviour of the spectral measure. If spectral density is non-increasing, then its asymptotic behaviour can be calculated directly. Note that if $\psi_{ij}$ is differentiable for $r > 0$ and  $\psi'_{ij}$ is concave, then $f_{ij}$ is monotonically decreasing function, see  \cite{askey1975some}. 

		\bibliographystyle{plain}	
		\bibliography{bibliography}
	\end{document}